\documentclass[sort&compress,review]{elsarticle}

\usepackage{a4}
\usepackage[utf8x]{inputenc}
\usepackage[english]{babel}
\usepackage[T1]{fontenc}
\usepackage{amsmath,amsthm}
\usepackage{amsfonts,amssymb}
\usepackage[nottoc, notlof, notlot]{tocbibind}
\usepackage{array}

\usepackage{float}
\usepackage{stmaryrd}
\usepackage{a4wide}
\usepackage{color} 
      \definecolor{bleu_sombre}{rgb}{0,0,0.6}  
      \definecolor{rouge_sombre}{rgb}{0.8,0,0}
      \definecolor{vert_sombre}{rgb}{0,0.6,0}
\usepackage[plainpages=false,colorlinks,linkcolor=bleu_sombre,citecolor=rouge_sombre,urlcolor=vert_sombre,breaklinks,backref=page]{hyperref} 
\usepackage{etex}
\usepackage{tikz}
\usepackage{pgfkeys}
\usepackage{pgffor}
\usepackage{pgfcalendar}
\usetikzlibrary{arrows,calc,matrix,topaths,positioning,scopes,shapes,decorations}

 \usepackage{mathpazo}
\usepackage{microtype} 
\usepackage{natbib}

\DeclareMathOperator{\Ima}{Im}
\DeclareMathOperator{\Ker}{Ker}

\DeclareMathOperator{\Hom}{Hom}
\DeclareMathOperator{\ds}{s^{-1}}
\DeclareMathOperator{\s}{s}

\DeclareMathOperator{\Tw}{Tw}

\DeclareMathOperator{\Ba}{\mathcal{B}}

\newtheorem*{rep@theorem}{\rep@title}
\newcommand{\newreptheorem}[2]{%
\newenvironment{rep#1}[1]{%
 \def\rep@title{#2 \ref{##1}}%
 \begin{rep@theorem}}%
 {\end{rep@theorem}}}

\newtheorem*{rep@prop}{\rep@title}
\newcommand{\newrepprop}[2]{%
\newenvironment{rep#1}[1]{%
 \def\rep@title{#2 \ref{##1}}%
 \begin{rep@prop}}%
 {\end{rep@prop}}}
\makeatother

\begin{document}


\theoremstyle{definition}
\newtheorem{de}{Definition}[section]
\theoremstyle{remark}
\newtheorem{rmq}[de]{Remark}
\newtheorem{exple}{Example}
\theoremstyle{plain}
\newtheorem{theo}[de]{Theorem}
\newtheorem{prop}[de]{Proposition}
\newtheorem{lem}[de]{Lemma}
\newtheorem{cor}[de]{Corollary}
\newrepprop{prop}{Proposition}
\newtheorem{properties}[de]{Properties}
\newreptheorem{theo}{Theorem}
\numberwithin{equation}{section}

\newtheorem*{cor*}{Corollary}

\newcommand{\cob}{\Omega C_* (X)}
\newcommand{\ccob}{\Omega^2 C_* (X)}
\newcommand{\Hopf}{\mathcal{H}}
\newcommand{\ot}{\otimes} 
\newcommand{\diskf}{\mathcal{C}^f_2} 
\newcommand{\disk}{\mathcal{C}_2}
\newcommand{\cobH}{\Omega \mathcal{H}}
\newcommand{\exd}{\widetilde{\sigma}} 
\newcommand{\tra}{\mathcal{T}}
\newcommand{\Zdeux}{\mathbb{Z}/2\mathbb{Z}}
\newcommand{\os}{\Omega \Sigma}
\newcommand{\La}{\Lambda}

\newcommand{\ov}[1]{\overline{#1}}
\newcommand{\OS}[1]{\Omega^{#1} \Sigma^{#1}}
\newcommand{\uv}[1]{\underline{#1}}
\newcommand{\wid}[1]{\widetilde{#1}}

\newcommand{\att}[1]{\textcolor{red}{#1}}
\newcommand{\note}[1]{\marginpar{\begin{footnotesize}\att{#1}\end{footnotesize}}}

\begin{frontmatter}

\title{{Homotopy BV-algebra structure on the double cobar construction}}

\author[rvt]{Alexandre Quesney}
\ead[rvt]{alexandre.quesney@univ-nantes.fr}
\address[rvt]{Laboratoire de Math\'ematiques Jean Leray - Université de Nantes, 2 rue de la Houssini\`ere - BP 92208 F-44322 Nantes Cedex 3, France.}

\begin{abstract}
We show that the double cobar construction, $\ccob$, of a simplicial set $X$ is a homotopy BV-algebra if $X$ is a double suspension, or if $X$ is 2-reduced and the coefficient ring contains the field of rational numbers $\mathbb{Q}$. Indeed, the Connes-Moscovici operator defines the desired homotopy BV-algebra structure on $\ccob$ when the antipode $S:\cob\to\cob$ is involutive. We proceed by defining a family of obstructions $O_n:\wid{C}_*(X)\to \wid{C}_*(X)^{\ot n}$, $n\geq 2$ by computing $S^2-Id$. When $X$ is a suspension, the only obstruction remaining is  $O_2:=E^{1,1}-\tau E^{1,1}$ where $E^{1,1}$ is the dual of the $\smile_1$-product.  When $X$ is a double suspension the obstructions vanish.
\end{abstract}

\end{frontmatter}

\tableofcontents

\section*{Introduction}
\addcontentsline{toc}{section}{Introduction}

Adams' cobar construction provides a model of the loop space of a $1$-connected topological space \cite{Adams}. 
The cobar construction is a functor from differential graded coalgebras to differential graded algebras;  
for iteration, a coproduct is thus needed.
For a $1$-reduced simplicial set $X$, 
Baues defined \cite{Bauesgeom} a DG-bialgebra structure on its first cobar construction $\cob$.
The resulting double cobar construction is an algebraic model for the double loop space. \\
We show that the double cobar construction, $\ccob$, of a simplicial set $X$ is a homotopy BV-algebra (a homotopy G-algebra in the sense of Gerstenhaber-Voronov \cite{GerstVoronov} together with a degree one operator) if $X$ is a double suspension, or if $X$ is 2-reduced and the coefficient ring contains the field of rational numbers $\mathbb{Q}$.

Baues' coproduct on $\cob$ is equivalent to a  homotopy G-coalgebra structure $\{E^{k,1}\}_{k\geq 1}$ on the DG-coalgebra $C_*(X)$, that is a family of operations 
\[
 E^{k,1}: \wid{C}_*(X)\to  \wid{C}_*(X)^{\ot k}\ot \wid{C}_*(X), ~~{k\geq 1},
\]
satisfying some relations.  This corresponds also to the coalgebra structure of $C_*(X)$ over the second stage filtration operad $F_2\chi$ of the surjection operad $\chi$ given in \cite{McClureSmithmulti,BergerFresse}, see Section \ref{section hGc}.
Since a bialgebra structure determines the antipode (whenever it exists), the antipode 
\[
 S:\cob\to \cob
\]
 on the cobar construction is then determined by the homotopy G-coalgebra structure on $C_*(X)$.

As a model for the chain complex of  double loop spaces, the double cobar construction is expected to be a Batalin-Vilkovisky algebra up to homotopy. 
Indeed, it is well-known that the circle action on the double loop space  $\Omega^2X$ by rotating the equator defines a BV operator; thereby the homology $H_*(\Omega^2 X)$ is a Batalin-Vilkovisky algebra \cite{Getzler}.

The cobar construction $\Omega \Hopf$ of an involutive Hopf algebra turns out to be the underlying complex in the Hopf-cyclic Hochschild cohomology of $\Hopf$, \cite{ConnesMosc}. 
The cyclic operator requires an involutive antipode on the underlying bialgebra $\Hopf$.
Assuming that the antipode is involutive, Menichi proved \cite{MenichiBV}, that for a unital (ungraded) Hopf algebra $\Hopf$, the Connes-Moscovici  operator induces a Batalin-Vilkovisky algebra on the homology of the cobar construction $H_*(\cobH)$.  

By defining a family of operations 
\[
 O_n:\wid{C}_*(X)\to \wid{C}_*(X)^{\ot n}, ~~n\geq 2,
\]
we provide a criterion for the involutivity of $S$ in terms of the operations 
\[
 E^{k,1}: \wid{C}_*(X)\to  \wid{C}_*(X)^{\ot k}\ot \wid{C}_*(X). 
\]

By Kadeishvili's work \cite{Kadeishvili} the double cobar construction is a homotopy G-algebra, it is endowed with a family of operations 
\[
 E_{1,k}: \ccob \ot \big(\ccob\big)^{\ot k} \to\ccob,~~k\geq 1,
\]
satisfying some relations. 
In particular, the following bracket 
\[
 \{a;b\}=E_{1,1}(a\ot b)- (-1)^{(|a|+1)(|b|+1)}E_{1,1}(b\ot a), ~~a,b\in\ccob,
\]
together with the DG-product of $\ccob$, induces a Gerstenhaber algebra structure on the homology $H_*(\ccob)$.
The vanishing of the operations 
\[
O_n:\wid{C}_*(X)\to \wid{C}_*(X)^{\ot n}, ~~n\geq 2,
\]
 gives a sufficient condition for the extension of the homotopy G-algebra structure on $\ccob$ given by Kadeishvili to the homotopy BV-algebra structure whose BV-operator is the  Connes-Moscovici operator.
We establish the following for a general homotopy G-coalgebra,
\begin{repprop}{propobstruction}
Let $(C,d,\nabla_C,E^{k,1})$ be a homotopy G-coalgebra. 
\begin{enumerate}
 \item The cobar construction $\Omega C$ is an involutive DG-Hopf algebra if and only if  all the obstructions $O_n:\wid{C}\to \wid{C}^{\ot n}$ defined in \eqref{invHcoal2} 
for $n\geq 2$ are zero.
\item Let $C$  be $2$-reduced i.e. $C_0=R$ and $C_1=C_2=0$.
If all the obstructions $O_n:\wid{C}\to \wid{C}^{\ot n}$  are zero, then the double cobar construction  $\Omega^2C$ is a homotopy BV-algebra given by the Connes-Moscovici operator.
\end{enumerate}
\end{repprop}
We apply this criterion to the homotopy G-coalgebra $C_*(X)$ of a $1$-reduced simplicial set $X$. We show that, when $\Sigma X$ is a simplicial suspension, the family of operations
\[
O_n:\wid{C}_*(\Sigma X)\to \wid{C}_*(\Sigma X)^{\ot n} 
\]
 reduces to 
\[
  O_2:\wid{C}_*(\Sigma X)\to \wid{C}_*(\Sigma X)\ot\wid{C}_*(\Sigma X).
 \]
 This operation $O_2$ is the deviation from the cocommutativity of the operation 
\[
E^{1,1}:\wid{C}_*(\Sigma X)\to \wid{C}_*(\Sigma X)\ot\wid{C}_*(\Sigma X). 
\]
In fact, $E^{1,1}$ is the only non-trivial operation of the family $\{E^{k,1}\}_{k\geq 1}$ defining the homotopy G-algebra structure on $C_*(\Sigma X)$, see Propositions \ref{susp1} and \ref{obstructionSuspension1}.
However, in the case of a double simplicial suspension $\Sigma^2 X$, the operation $E^{1,1}$ is also trivial, and all the obstructions 
\[
 O_n:\wid{C}_*(\Sigma^2 X)\to \wid{C}_*(\Sigma^2 X)^{\ot n}, ~~n\geq 2,
\]
are zero. 
As a consequence, the cobar construction $\Omega C_*(\Sigma^2 X)$ is the free tensor DG-algebra with the shuffle coproduct. Thus we have,
\begin{reptheo}{DoubleSuspension}
 Let $\Sigma^2 X$ be a  double suspension. Then:
\begin{itemize}
 \item the homotopy G-coalgebra structure on $C_*(\Sigma^2 X)$ corresponding to Baues' coproduct 
\[
\nabla_0:\Omega C_*(\Sigma^2 X) \to \Omega C_*(\Sigma^2 X)\ot  \Omega C_*(\Sigma^2 X), 
\]
 has trivial higher operations i.e. $E^{1,k}=0$ for $k\geq 1$;
 \item the double cobar construction $\Omega^2 C_*(\Sigma^2 X)$ is a homotopy BV-algebra with the Connes-Moscovici operator as BV-operator.
\end{itemize}
\end{reptheo}

On the other hand, if the ground ring contains the field of rational numbers $\mathbb{Q}$, we deform (see Section \ref{sectionQ} for the precise statement) $(\cob,\nabla_0,S)$ into a cocommutative DG-Hopf algebra $(\cob,\nabla'_0,S')$. 
Therefore, $(\cob,\nabla'_0,S')$ has an involutive antipode and the homotopy BV-algebra structure we consider on the deformed double cobar construction  $\Omega (\cob,\nabla'_0,S')$ follows. Thus, we have
\begin{reptheo}{CorBV}
 Let $X$ be a $2$-reduced simplicial set. Then the double cobar construction $\Omega (\cob,\nabla'_0,S')$ over $R\supset\mathbb{Q}$ coefficients is a homotopy BV-algebra  with BV-operator the Connes-Moscovici operator.
\end{reptheo}

The paper is organized as follows.

In the first section we review background materials on the bar-cobar constructions and Hopf algebras.\\
The second section is devoted to the structures of homotopy G-algebras \cite{Kadeishvili} and homotopy BV-algebras \cite{MenichiBV} on the cobar construction. These first two sections fix notations and sign conventions.\\
In the third section we define the family of operations $O_n:\wid{C}\to \wid{C}^{\ot n}$, $n\geq 2$,
on a homotopy G-coalgebra $(C,E^{k,1})$.\\
In the section 4 we give applications :\\
In the subsection 4.1 we set the convention for the homotopy G-coalgebra structure on $C_*(X)$. We compare the homotopy G-coalgebra structure on $C_*(X)$ coming from Baues' coproduct with the action of the surjection operad given in \cite{McClureSmithmulti,BergerFresse}.\\
Subsections 4.2 and 4.3 give applications to simplicial suspensions. 
In the case of single suspension $\Sigma X$ the family of operations $O_n:\wid{C}_*(\Sigma X)\to \wid{C}_*(\Sigma X)^{\ot n}$, $n\geq 2$ reduces to $O_2$.
 We show that, for a double suspension, this last obstruction $O_2$ vanishes. \\
The last subsection 4.4 is devoted to the rational case. We prove that the double cobar construction of a $2$-reduced simplicial set is a homotopy BV-algebra.\\
In Appendix we recall and specify some facts about the Hirsch and the homotopy G-algebras. 
In particular, we make explicit the signs related to our sign convention.

\section{Notations and preliminaries}\label{section1}

\subsection{Conventions and notations}
Let $R$ be a commutative ring.
A graded $R$-module $M$ is a family of $R$-modules $\{M_n\}$ where indices $n$ run through the integers.
The degree of $a \in M_n$ is denoted by $|a|$, so here $|a|=n$.
The $r$-suspension $s^{r}$ is defined by $(s^{r}M)_n=M_{n-r}$. \\
Algebras (respectively coalgebras) are understood as associative algebras (respectively coassociative coalgebras).\\
A unital $R$-algebra $(A,\mu,\eta)$ is called augmented if there is an algebra morphism $\epsilon:A\to R$. We denote by $\ov{A}$ the augmentation ideal $\Ker(\epsilon)$. \\
 For a coalgebra $(C,\nabla)$ the $n$-iterated coproduct is denoted by
\[
\nabla^{(n)}:C\to C^{\ot n+1} 
\]
 for $n\geq 1$.
We use the Sweedler notation, 
\[
\nabla(c)=c^1\ot c^2  ~~~\text{and}
~~~\nabla^{(n)}(c)=c^1\ot c^2 \ot \cdots \ot c^{n+1},
\]where we have omitted the sum.
A counital $R$-coalgebra $(C,\epsilon)$ is called coaugmented if there is a coalgebra morphism $\eta:R\to C$. 
We denote by $\ov{C}=\Ker(\epsilon)$ the reduced coalgebra with the reduced coproduct 
\[
\overline{\nabla}(c)={\nabla}(c)-c\ot 1 -1\ot c.  
\]

A (co)algebra $A$ is called connected if it is both (co)augmented and $A_n=0$ for $n\leq -1$ and $A_0\cong R$.\\
A (co)algebra, $A$ is called $n$-reduced if it is both connected and $A_k=0$ for $1\leq k\leq n$.\\
A coaugmented coalgebra $C$ is called conilpotent if the following filtration 
\begin{align*}
F_0C&:=R \\
F_rC&:=R\oplus \{c\in \ov{C}| \nabla^n(c)=0, n\geq r \}~~\text{for}~~ r\geq 1,
\end{align*}
is exhaustive, that is $C=\bigcup_r F_rC$.
\subsection{The bar and cobar constructions}

We refer to \cite[Chapter 2]{LV} for the background materials related to the bar and cobar constructions.

The cobar construction is a functor 
\begin{align*} 
 \Omega : DGC_1 &\to DGA_0 \\
  (C,d_C,\nabla_C) &\mapsto \Omega C =(T(\ds \ov{C}), d_{\Omega})
 \end{align*}
from the category of $1$-connected DG-coalgebras to the category of connected DG-algebras.
Here, $T(\ds \ov{C})$ is the free tensor algebra on the module $\ds \ov{C}$ and $d_{\Omega}$ is the unique derivation such that $d_{\Omega}(\ds c)=-\ds d_{\ov{C}}  (c)+ (\ds\ot \ds )\ov{\nabla}_{C} (c)$ for all $c\in  \ov{C}$.

The bar construction is a functor 
\begin{align*} 
 \Ba : DGA_0 &\to DGC_c \\
  (A,d_A,\mu_A) &\mapsto \Ba A =(T^c(\s \ov{A}), d_{\Ba})
 \end{align*}
from the category of connected DG-algebras to the category of conilpotent DG-coalgebras.
Here, $T^c(\s \ov{A})$ is the cofree tensor coalgebra on the module $\s \ov{A}$ and $d_{\Ba}$ is the unique coderivation with components 
\begin{center}
\begin{tikzpicture} [>=stealth,thick,draw=black!50, arrow/.style={->,shorten >=1pt}, point/.style={coordinate}, pointille/.style={draw=red, top color=white, bottom color=red}]
\matrix[row sep=9mm,column sep=16mm,ampersand replacement=\&]
{
  \node (a1){$T^c(\s \ov{A})$}; \& \node (b1) {$\s \ov{A}\oplus (\s \ov{A})^{\ot 2}$} ;\& \node (c1) {$\s \ov{A}$.} ;\\
}; 
\path
	  
   	 (a1)     edge[above,arrow,->>]      node {}  (b1)
	 (b1)     edge[above,arrow,->]      node[yshift=0.3cm] {\begin{scriptsize}$-\s d_{\ov{A}}\ds+ \s \mu_{\ov{A}}(\ds \ot \ds) $	                                                                                      \end{scriptsize}}  (c1)
     ; 
\end{tikzpicture}
\end{center}

We recall the bar-cobar adjunction,
\begin{theo}\cite[Theorem 2.2.9]{LV}\label{adjun}
For every augmented DG-algebra $\La$ and every conilpotent DG-coalgebra $C$ there exist natural bijections
\[
\Hom_{\text{DG-Alg}}(\Omega C,\La) \cong \Tw(C;\La) \cong \Hom_{\text{DG-Coalg}}(C,\Ba\La).
\]
\end{theo}
The set $\Tw(C;\La)$  of twisting cochains from $C$ to $\La$ is the set of degree $1$ linear maps $f:C\to \La$ verifying the twisting condition: $\partial f:=df +fd =- \mu_{\La}(f\ot f)\nabla_C$.

\subsection{Hopf algebras}

\begin{de}\label{convolution}
For a  DG-bialgebra $(\mathcal{H},d,\mu,\eta,\nabla,\epsilon)$ an \textbf{antipode} $S:\mathcal{H} \to \mathcal{H}$ is a chain map which is the inverse of the identity in the convolution algebra $\Hom(\mathcal{H},\mathcal{H})$; the convolution product being $f\smile g =\mu(f\ot g)\nabla$. Explicitly, $S$ satisfies 
\begin{equation*}
S(a^1)a^2=\eta\epsilon(a)=a^1S(a^2) 
\end{equation*}
for all $a\in \mathcal{H}$. 
\end{de}
\begin{de}
 A \textbf{DG-Hopf algebra} $(\Hopf,d,\mu,\eta,\nabla,\epsilon,S)$ is a DG-bialgebra $\Hopf$ endowed with an antipode $S$.
If moreover, the antipode is involutive (i.e. $S^2=id$) we call $\Hopf$ an \textbf{involutive DG-Hopf algebra}.
\end{de}
An antipode satisfies the following properties.
 \begin{prop}\label{antip properties}\cite[Proposition 4.0.1]{Sweedler}
\renewcommand{\theenumi}{\roman{enumi}}
\begin{enumerate}
\item $S(a^2)\otimes S(a^1) =(-1)^{|a^1||a^2|} S(a)^1\otimes S(a)^2$  ~~\text{(coalgebra antimorphism)}. \label{eqantipode2}
\item $S(\eta(1))=\eta(1)$ ~~\text{(unital morphism)}.
\item $S(ab)=(-1)^{|a||b|}S(b)S(a)$   ~~\text{(algebra antimorphism)}.\label{eqantipode3}
\item $\epsilon(S(a))=\epsilon(a)$  ~~\text{(counital morphism)}. \label{eqantipode4}
\item The following equations are equivalent:
\begin{enumerate}
\item $S^2=S\circ S=id$;
 \item  $S(a^2)a^1=\eta\epsilon(a)$;
\item $a^2S(a^1)=\eta\epsilon(a)$.
\end{enumerate}
\item If $\mathcal{H}$ is commutative or cocommutative, then $S^2=id$.
\end{enumerate}
\end{prop}

\section{Homotopy structures on the cobar construction}

\subsection{Homotopy G-algebra on the cobar construction}
We present the homotopy G-algebra structure on the cobar construction of a $1$-reduced DG-bialgebra given in \cite{Kadeishvili}. The homotopy G-algebras are also known as the Gerstenhaber-Voronov algebras defined in \cite{GerstVoronov}.
We define the Hirsch algebras and the homotopy G-algebras via the bar construction.

\begin{de}
A \textbf{Hirsch algebra} $\La$ is the data of a connected DG-algebra $(\La,d,\mu_{\La})$ together with a map $\mu:\Ba\La \ot \Ba\La \to \Ba\La$ making $\Ba\La$ into an associative unital  DG-bialgebra. 
\end{de}
For a connected DG-algebra $\La$ a DG-product $\mu:\Ba\La \ot \Ba\La \to \Ba\La$ corresponds to a twisting cochain $\wid{E}\in \Tw(\Ba\La\ot \Ba\La, \La)$, cf. Theorem \ref{adjun}.  After a correct use of desuspensions (see sign convention below), the latter is a family of operations $\{E_{i,j}\}_{i,j\geq 1}$, $E_{i,j}:\ov{\La}^{\ot i} \ot \ov{\La}^{\ot j} \to \ov{\La}$ satisfying some relations, see \cite{Kadeishvili}.
We denote by $\mu_E:\Ba\La \ot \Ba\La \to \Ba\La$ such a DG-product.\\

A Hirsch algebra is a particular $B_{\infty}$-algebra whose underlying $A_{\infty}$-algebra structure is a DGA structure. 
Recall that the bar construction $\Ba \La$ of a DG-algebra $\La$ is the cofree tensor coalgebra $T^c(s\ov{\La})$ together with the derivation  induced by the differential and the product of $\La$.
\begin{center}
\begin{tabular}[h]{|ll|ll|}
 \hline
 & $\La$ & &$T^c(s{\ov{\La}})$ \\ \hline
DGA & $(d_{\La},\mu_{\La})$ & DGC &  $d$ induced by $d_{\La}$ and $\mu_{\La}$\\ \hline
$A_{\infty}$ & $\{\mu_k\}_{k\geq 0}$ & DGC &  $d$ induced by $\mu_k$ for $k\geq 0$\\ \hline
Hirsch & $(d_{\La},\mu,\{E_{i,j}\}_{i,j\geq 1})$& DG-bialgebra &  $d$ induced by $d_{\La}$ and $\mu$ ; DG-product $\mu_E$\\ \hline
$B_{\infty}$ &  $(\{\mu_k\}_{k\geq 0},\{E_{i,j}\}_{i,j\geq 1})$ &  DG-bialgebra &  $d$ induced by $\mu_k$ for $k\geq 0$ ; DG-product $\mu_E$\\ \hline 
 \end{tabular}
\end{center}

\begin{de}\label{def: homotopy G}
 A \textbf{homotopy G-algebra} $(\La,d_{\La},\mu_{\La},\{E_{1,j}\}_{j\geq 1})$
 is a Hirsch algebra $(\La,d_{\La},\mu_{\La},\{E_{i,j}\}_{i,j\geq 1})$ such that $E_{i,j}=0$ for $i\geq 2$.  
\end{de}

\subsubsection*{Sign convention} 
Let  $\wid{E}\in\Tw(\Ba\La\ot \Ba\La,\La)$ be a twisting cochain.
 Its $(i,j)$-component is $\wid{E}_{i,j}:(\s\ov{\La})^{\ot i}\ot (\s\ov{\La})^{\ot j} \to \ov{\La}$. We denote by  $E_{i,j}:\ov{\La}^{\ot i}\ot \ov{\La}^{\ot j} \to \ov{\La}$ the component 
\begin{multline}\label{conv}
 E_{i,j}(a_1\ot ...\ot a_i\ot a_{i+1} \ot...\ot a_{i+j})\\
:=(-1)^{\sum_{s=1}^{i+j-1}|a_s|(i+j-s)} \wid{E}_{i,j}(\s^{\ot i}\ot \s^{\ot j})(a_1\ot ...\ot a_i\ot a_{i+1} \ot...\ot a_{i+j})\\
=\wid{E}_{i,j}(\s a_1\ot ...\ot \s a_i\ot \s a_{i+1} \ot...\ot \s a_{i+j}).
\end{multline}
The degree of $E_{i,j}$ is $i+j-1$.
We denote $E_{i,j}( (a_1\ot ...\ot a_{i}) \ot (b_1\ot ...\ot b_{j}))$ by $E_{i,j}(a_1,...,a_{i};b_1,...b_{j})$. 

Let $B$ be a $1$-reduced DG-bialgebra.
For $a\in B$ and $\overline{b}:=\ds b_1\ot ... \ot \ds b_s \in \Omega B$, we set 
\[
 a\diamond \overline{b}=\ds (a^1b_1)\ot ...\ot \ds(a^sb_s).
\]
\begin{prop}{\cite{Kadeishvili}}
 Let $B$ be a $1$-reduced DG-bialgebra. Then the cobar construction $(\Omega B,d_{\Omega})$ together with the operations  $E_{1,k}$, $k\geq 1$ defined by 
\begin{multline}\label{E1k}
 E_{1,k}(\ds a_1\ot ...\ot \ds a_n;\overline{b}_1,...,\overline{b}_k)=\\ \sum_{1\leq i_1\leq i_2\leq ...\leq i_k\leq n} \pm \ds a_1\ot ...\ot \ds a_{i_1-1}\ot a_{i_1}\diamond\overline{b}_1\ot ...\ot \ds a_{i_k-1}\ot a_{i_k}\diamond\overline{b}_k \ot \ds a_{i_k+1}\ot ...\ot \ds a_{n}
\end{multline} when $n\geq k$ and zero when $n<k$ 
is a homotopy G-algebra.
\end{prop}

In particular the operation $E_{1,1}:\ov{\Omega B} \ot \ov{\Omega B} \to \Omega B$ is given by 
\begin{multline}\label{E11}
 E_{1,1}(\ds a_1 \ot ...\ot\ds a_m ; \ds b_1 \ot ... \ot \ds b_n)\\
:= \sum_{l=1}^{m} (-1)^{\gamma}\ds a_1 \ot... (a_l^1\cdot b_1)\ot ...\ot \ds(a_l^n\cdot b_n)\ot ...\ot \ds a_m
\end{multline}
with
\begin{align*}
 \gamma &=\sum_{u=1}^{n}|b_u|\left(\sum_{s=l+u}^{m}|a_s|+m-l-u+1\right)+  
\kappa_n(a_l)+ \sum_{u=2}^{n}(|a_l^u|-1)\left(\sum_{s=1}^{u-1}|b_s|-u+1 \right)\\
&+ \sum_{s=1}^{n}(|a_l^s|-1)(2n-2s+1) +\sum_{s=1}^{n-1}(|a_l^s|+|b_s|)(n-s),
 \end{align*}
where
\begin{align*}
 \kappa_n(a):=\begin{cases}
              \sum_{1\leq 2s+1\leq n}|a^{2s+1}|     & \text{if $n$ is even} ;\\  
	      \sum_{1\leq 2s\leq n}|a^{2s}|     & \text{if $n$ is odd} .\\
	      \end{cases}
\end{align*}
We postpone the construction of this operation to Appendix.

\subsection{Homotopy BV-algebra on the cobar construction}
We define a homotopy BV-algebra as a homotopy G-algebra (in the sense of Gerstenhaber-Voronov \cite{GerstVoronov}) together with a BV-operator. 
The main example, that we will discuss in details, is the cobar construction of a $1$-reduced involutive DG-Hopf algebra $\Hopf$.
The BV-operator on $\Omega \Hopf$ is the Connes-Moscovici operator defined in \cite{ConnesMosc}.

\begin{de}
 A \textbf{homotopy BV-algebra} $\La$ is a homotopy G-algebra together with a degree one DG-operator $\Delta:\La \to \La$ subject to the relations:
\begin{align*}
\Delta^2&=0;\\
 \{a;b\}&=(-1)^{|a|}\bigl( \Delta(a\cdot b)-\Delta(a)\cdot b -(-1)^{|a|}a\cdot \Delta(b)\bigr)\\
         &~~~~~ + d_{\La}H(a;b)+H(d_{\La}a;b) +(-1)^{|a|} H(a;d_{\La}b)~~\text{for all} ~~a,b\in \La,
\end{align*}     
where $\{-;-\}$ is the Gerstenhaber bracket $\{a;b\}=E_{1,1}(a; b)- (-1)^{(|a|+1)(|b|+1)}E_{1,1}(b; a)$ and $H:\La \ot \La \to \La$ is a degree $2$ linear map.
\end{de} 

A straightforward calculation shows that:
\begin{align*}
 \Delta(\{a;b\})&=\{\Delta(a);b\}+(-1)^{|a|+1}\{a;\Delta(b)\} - \partial \overline{H}(a;b),
\end{align*}
where 
\[
 \overline{H}(a;b):=\Delta H(a;b)+ H(\Delta(a);b) + (-1)^{|a|}H(a;\Delta(b))~~ \forall a,b\in \La.
\]
Therefore $\Delta$ is a derivation for the bracket $\{-;-\}$ 
if $\Delta$  is a derivation for $H$, or more generally 
if $\Delta$  is a derivation for $H$ up to a homotopy.
\\
In the example we consider hereafter, the homotopy $H$ is itself antisymmetric, more precisely, we have
$H({a};{b}):=H_1({a};{b})-(-1)^{(|{a}|+1)(|{b}|+1)}H_1({b};a)$. 
It turns out that the operator $\Delta$ is not a derivation for $H$ in general.
However we do not know yet if there exists a homotopy $H'$ such that $\overline{H}=\partial H'$.

A. Connes and H. Moscovici defined in \cite{ConnesMosc} a boundary map on the cobar construction of an involutive Hopf algebra $\Hopf$. 
This operator, hereafter called the Connes-Moscovici operator, induces a BV-operator on the homology $H_*(\Omega \Hopf)$.
 More precisely,
let $(\Hopf,d,\mu,\eta,\nabla,\epsilon,S)$ be an involutive DG-Hopf algebra, that is with an involutive antipode $S$. 
With our sign convention, the Connes-Moscovici operator 
\[
 \Delta: \cobH \to \cobH
\]
is zero on both  $R=(\ds\ov{\Hopf})^{\ot 0}$ and $(\ds\ov{\Hopf})^{\ot 1}$, and is given on the $n$-th component $(\ds\ov{\Hopf})^{\ot n}$, $n\geq 2$ by:
\begin{equation}\label{BVop}
 \Delta(\ds a_1\ot \ds a_2\ot...\ot \ds a_n)=
\sum_{i=0}^{n-1}(-1)^{i}\pi_n\tau_n^{i}(\ds a_1\ot \ds a_2\ot... \ot  \ds a_n),
\end{equation}
where $\tau_n$ is the cyclic permutation
\begin{align*}
\tau_n(\ds a_1\ot \ds a_2 \ot ...\ot  \ds a_n):= (-1)^{(|a_1|-1)\left(\sum_{i=2}^{n}|a_i|-1\right)}\ds a_2 \ot\ds a_3 \ot...\ot \ds a_n\ot \ds a_1,
\end{align*} and 
\begin{align*}
\pi_n:=(\ds\mu)^{\ot n-1}\tau_{n-1,n-1}(S^{\ot n-1}\ot 1^{\ot n-1})((\tau \nabla)^{(n-2)}\ot1^{n-1})\s^{\ot n},\\
\tau_{n,n}(a_1\ot a_2 \ot ...\ot a_n\ot b_1 \ot ...\ot b_n):=(-1)^{\sum_{i=1}^{n-1}|b_i|(\sum_{j=i+1}^n|a_j|)} a_1\ot  b_1 \ot  a_2\ot b_2\ot... \ot a_n \ot b_n.
\end{align*}
More explicitly, 
\begin{equation}
  \Delta(\ds a_1\ot \ds a_2\ot...\ot \ds a_n)=
\sum_{k=1}^{n}\pm \ds S(a_k^{n-1})a_{k+1}\ot \ds S(a_{k}^{n-2})a_{k+2}\ot... \ot  \ds S(a_k^{1})a_{k-1}.
\end{equation}
The involutivity of the antipode of $\Hopf$ makes $\Delta$ into a square zero chain map.
L. Menichi proved \cite[Proposition 1.9]{MenichiBV} that for a unital (ungraded) Hopf algebra $\Hopf$, the
Connes-Moscovici operator induces a Batalin-Vilkovisky algebra on the homology of the
cobar construction $H_*(\Omega\Hopf)$.
In our context,
\begin{prop}{\cite{MenichiBV}}\label{PropBV}
 Let $\Hopf$ be a $1$-reduced involutive DG-Hopf algebra. Then the cobar construction $(\Omega \Hopf,d_{\Omega})$ is a homotopy BV-algebra whose the BV-operator is defined in \eqref{BVop}.
\end{prop}
\begin{proof}
 Let us define Menichi's homotopy $H(\ov{a};\ov{b}):=H_1(\ov{a};\ov{b})-(-1)^{(|\ov{a}|+1)(|\ov{b}|+1)}H_1(\ov{b};\ov{a})$. 
\begin{multline}
 H_1(\ds a_1\ot... \ot \ds a_m ;\ds b_1\ot... \ot \ds b_n)\\
:= \sum_{1\leq j\leq p\leq m-1}(-1)^{\xi_j+n+1} \pi_{m+n-1}^s\tau_{m+n-1}^{s,n+m-1-j}\rho_{n+m}^{(p-j+1)}(\ds a_1\ot...\ot \ds a_m \ot \ds b_1\ot... \ot \ds b_n)
\end{multline}
with 
\begin{equation*}
 \xi_j=\left\{\begin{array}{rl}
       \displaystyle \sum_{s=1}^{m} |a_s| & \text{for}~ j=1;\\
	\displaystyle \sum_{s=m-j+1}^{m} |a_s| & \text{for}~ j>1.
       \end{array}
\right.
\end{equation*}
Where,
\begin{align*}
\pi_m^s&= \pi_m(\ds)^{\ot m}\\
\tau_m^{s,i}&:=(\tau_m^{s})^i \\
\tau_m( a_1\ot  a_2 \ot ...\ot a_m)&:= (-1)^{|a_1|\left(\sum_{i=2}^{n}|a_i|\right)} a_2 \ot a_3 \ot...\ot  a_m\ot a_1\\
\rho_{m+1}^{(i)}&:= (1^{\ot i-1}\ot \mu\ot 1^{\ot m-i})(1^{\ot i-1}\ot \tau_{m-i+1,1})s^{\ot m+1}\\
\tau_{k,1}( a_1\ot ...\ot  a_k \ot b)&:= (-1)^{|b|\left(\sum_{i=2}^{k}|a_i|\right)} a_1 \ot b\ot a_2 \ot...\ot  a_k.
\end{align*}
By Proposition  \cite[Proposition 1.9]{MenichiBV} we have
\begin{equation*}
 \{a;b\}=(-1)^{|a|}\bigl( \Delta(a\cdot b)-\Delta(a)\cdot b -(-1)^{|a|}a\cdot \Delta(b)\bigr) + d_{1}H(a;b)+H(d_{1}a;b) +(-1)^{|a|} H(a;d_{1}b)
\end{equation*}
on the cobar construction $(\Omega \Hopf, d_0+d_1)$, where $d_1$ is the quadratic part of the differential $d_{\Omega}$.
The operators involved in the above equation commute with $d_{0}$. Therefore we can replace $d_{1}$ by $d_{\Omega}$ in the previous equation. 
\end{proof}

\section[Involutivity of the antipode of $\Omega C$ in terms of homotopy G-coalgebra $C$]{Involutivity of the antipode of $\Omega C$ in terms of the homotopy G-coalgebra $C$}
A Hirsch coalgebra $C$ is the formal dual of a Hirsch algebra, that is it corresponds to a DG-coproduct $\nabla:\Omega C \to \Omega C \ot \Omega C$ making $\Omega C$ into a DG-bialgebra.
Baues' coproduct \cite{Bauesgeom} $\nabla_0: \cob \to \cob$ defined on the cobar construction of a simplicial set $X$ corresponds to a homotopy G-coalgebra structure $E^{k,1}$ on $C_*(X)$, that is a particular case of Hirsch coalgebra structure. 
With this example in mind we consider a homotopy G-coalgebra $(C,E^{k,1})$. 
Then there exists a unique antipode $S:\Omega C \to \Omega C$ on its cobar construction.
The purpose of this section is to give a criterion for the involutivity of the antipode $S$ in terms of the operations $E^{k,1}$. 
 This takes the form of a family of operations 
\[
O_n:\ds\ov{C}\to (\ds\ov{C})^{\ot n}, ~~n\geq 2.
\]
For convenience, we set $\wid{C}:=\ds\ov{C}$. Thus $O_n:\wid{C}\to \wid{C}^{\ot n}$, $n\geq 2$.
\begin{de}
A \textbf{Hirsch coalgebra} $C$ is the data of a $1$-reduced DG-coalgebra $(C,d,\nabla_C)$ together with a map $\nabla:\Omega C \to \Omega C \ot \Omega C$ making $\Omega C$ into a coassociative
 counital  DG-bialgebra. The corresponding operations on $C$ are denoted by $E^{i,j}:\wid{C}\to \wid{C}^{\ot i}\ot \wid{C}^{\ot j}$. The degree of $E^{i,j}$ is $0$.
 A \textbf{homotopy G-coalgebra} is a Hirsch coalgebra whose operations $E^{i,j}=0$ for $j\geq 2$.  
\end{de}
Let the cobar construction $(\Omega C,d,\mu_{\Omega},\nabla)$  be a DG-bialgebra. We can define the antipode 
\[
S:\Omega C \to \Omega C 
\]
by 
$S(\eta(1))=\eta(1)$ and for $\sigma\in \wid{C}_n$ with $n\geq 1$ by 
\begin{equation}\label{antpodeconst}
S(\sigma):=-\sigma -\mu_{\Omega}(S\ot 1)\overline{\nabla}(\sigma) 
\end{equation}
 which makes sense since 
\begin{equation*}
\overline{\nabla}(\sigma)\subset \bigoplus_{\substack{i+j=n\\ 0<i,j<n}} (\Omega C)_i\ot (\Omega C)_j.                          
\end{equation*}
Indeed, 
  $\mu_{\Omega}(S\ot 1)\nabla(\eta(1))=\eta\epsilon\eta(1)=\eta(1)$ gives immediately that $S(\eta(1))=\eta(1)$.  
 Moreover, since $\mu_{\Omega}(S\ot 1)\nabla(\sigma)=0$ for all $\sigma\in \wid{C}_n$ with $n\geq 1$, 
we have
\begin{align*}
\mu_{\Omega}(S\ot 1)\nabla(\sigma)&=\mu_{\Omega}(S\ot 1)\ov{\nabla}(\sigma)+ \mu_{\Omega}(S(\sigma)\ot \eta(1)+ S(\eta(1))\ot \sigma)\\
&=\mu_{\Omega}(S\ot 1)\ov{\nabla}(\sigma)+ S(\sigma)+\sigma=0.
\end{align*} 

\begin{rmq}
Equivalently, we can define $S$, see \cite[section 1.3.10 p.15]{LV},  as the geometric serie $(Id)^{\smile -1}=\sum_{n\geq 0} (\eta\epsilon-Id)^{\smile n}$ with $(\eta\epsilon-Id)^{\smile 0}=\eta\epsilon$ and where $\smile$ denote the convolution product from Definition \ref{convolution}. Note that this sum is finite when it is evaluated on an element. This presentation is combinatorial and gives for $[\sigma] \in (\ds \ov{C})_n$,
 \begin{equation}
  \begin{split}
     S([\sigma]) = \eta\epsilon([\sigma])+ (\eta\epsilon-Id)([\sigma])+ \mu_{\Omega}\overline{\nabla}([\sigma]) - \mu_{\Omega}^{(2)} \overline{\nabla}^{(2)}([\sigma])+...+(-1)^{n-1}\mu_{\Omega}^{(n)} \overline{\nabla}^{(n)}([\sigma]).
 \end{split}
\end{equation}
\end{rmq}

By the \ref{eqantipode3} of Proposition \ref{antip properties} the antipode is an algebra antimorphism that  is an algebra morphism from $(\Omega C,\mu_{\Omega })$ to $\Omega C_{(12)}:=(\Omega C,\mu_{{\Omega}_{(12)}}:=\mu_{\Omega }\tau)$. Moreover, it is also a DG-map, therefore it corresponds to a twisting cochain $F \in \Tw({C},\Omega C_{(12)})$. An antipode is determined by the underlying bialgebra structure. Here the latter is equivalent to a homotopy G-algebra structure $E^{k,1}$ on $C$. We make explicit the twisting cochain $F$  in terms of $E^{k,1}$.\\
We recall the notation $\wid{C}:=\ds\ov{C}$. We write $F^i:\wid{C}\to \wid{C}^{\ot i}$ for the $i$-th component of $F$.
The relation\footnote{We can also consider $\mu_{\Omega }(S\ot 1)\nabla = \eta\epsilon$. It gives equivalent $F^n$ but with a more complicated description because of the apparition of permutations. For example $F^3=-(E^{1,1}\ot 1)E^{1,1}- (\tau\ot 1)E^{2,1}$.}  $\mu_{\Omega }(1\ot S)\nabla = \eta\epsilon$ gives
\begin{align}
F^1&=-Id_{\wid{C}};\\
F^n&=\sum_{\substack{1\leq s \leq n-1 \\ n_1+...+n_s=n-1 \\n_i\geq 1}}(-1)^{s+1}(1^{\ot n_1+...+n_{s-1}}\ot E^{n_s,1})...(1^{\ot n_1}\ot E^{n_2,1})E^{n_1,1}, ~~n\geq 2. \label{Fn}
\end{align}
Let the operation $E^{n_i,1}:\wid{C}\to \wid{C}^{\ot n_i}\ot  \wid{C}$ be represented by the tree in Figure \ref{fig1}. 
Then the summands of the equation \eqref{Fn} are represented by the tree in Figure \ref{fig2}. 
\begin{center}
\begin{minipage}{.45\textwidth}
\begin{figure}[H]
\begin{tikzpicture}
[level distance=8mm,growth parent anchor=north,
every node/.style={anchor=north,rectangle,draw}
every child node/.style={anchor=south},
level 1/.style={sibling distance=30mm},
level 2/.style={sibling distance=5mm},
level 3/.style={sibling distance=15mm},
level 4/.style={sibling distance=5mm},
level 5/.style={sibling distance=10mm},
level 6/.style={sibling distance=5mm}
]
\coordinate 
 child { child {node (l2) {$1$}} child{node{$\cdots$}} child {node (r2) {$n_i$}} } 
child { child }
;
\end{tikzpicture} 
\caption{The operation $E^{n_i,1}$.}
\label{fig1}
\end{figure}
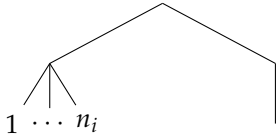
\end{minipage}
\begin{minipage}{.45\textwidth}
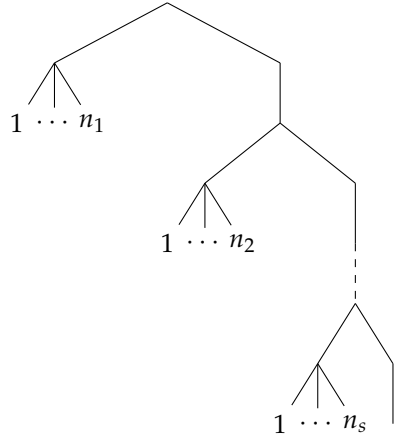
\begin{figure}[H]
\begin{tikzpicture}
[level distance=8mm,growth parent anchor=north,
every node/.style={anchor=north,rectangle,draw}
every child node/.style={anchor=south},
level 1/.style={sibling distance=30mm},
level 2/.style={sibling distance=5mm},
level 3/.style={sibling distance=20mm},
level 4/.style={sibling distance=5mm},
level 5/.style={sibling distance=10mm},
level 6/.style={sibling distance=5mm}
]
\coordinate
child { child {node (l2) {$1$}} child{node{$\cdots$}} child {node (r2) {$n_1$}} } 
child { child{
	      child { child{node{$1$}} child{node{$\cdots$}} child{node{$n_2$}} }
	      child {child[] { child[dashed]{} 
 { 
			    child  {child{node{$1$}} child{node{$\cdots$}} child{node{$n_s$}}}
			    child {child}}}}}}
;
\end{tikzpicture}
\caption{A summand of $F^n$}
\label{fig2}
\end{figure}
\end{minipage}
\end{center}
The three first terms are
$F^1=-Id_{\wid{C}}$, $F^2=E^{1,1}$ and $F^3=-(1\ot E^{1,1})E^{1,1}+ E^{2,1}$.

Formulated in these terms, the $n$-th component of the difference $S^2-Id$ is zero for $n=1$ and 
\begin{align}\label{invHcoal2}
O_n:=\sum_{s=1}^{n}\sum_{n_1+...+n_s=n}\mu_{\Omega_{(12)}}^{(s)}(F^{n_1}\ot ...\ot F^{n_s})F^s
\end{align}for $n\geq 2$.
The terms $(F^{n_1}\ot ...\ot F^{n_s})F^s$ have codomain $\wid{C}^{\ot n_1}\ot \wid{C}^{\ot n_2}\ot \cdots \ot \wid{C}^{\ot n_s}$. The $s$-iterated product $\mu_{\Omega_{(12)}}^{(s)}$ permutes these $s$ blocks as the permutation 
\begin{tikzpicture}
\matrix [matrix of math nodes,left delimiter={(},right delimiter={)}]
{
1 & 2 & \cdots & s-1 & s \\
s & s-1 & \cdots & 2 & 1 \\
};
\end{tikzpicture}
in $\mathbb{S}_s$, where $\mathbb{S}_s$ denotes the symmetric group on $s$ objects.
The two first terms we obtain are 
\[
  O_2=F^2-\tau F^2=E^{1,1}-\tau E^{1,1}
\]
and 
\begin{multline*}
O_3=(1+\tau_3)F^3+\tau^{2,1}(F^2\ot 1)F^2+\tau^{1,2}(1\ot F^2)F^2\\
= (\tau^{1,2}-1-\tau_3)(1\ot E^{1,1})E^{1,1}+\tau^{2,1}(E^{1,1}\ot 1)E^{1,1} + (1+\tau_3) E^{2,1}, 
\end{multline*}
 where the permutations are 
$\tau^{1,2}(a\ot b\ot c)=\pm b\ot c\ot a$, $\tau^{2,1}(a\ot b\ot c)=\pm c\ot a\ot b$ and $\tau_3(a\ot b\ot c)=\pm c\ot b \ot a$; the signs being given by the Koszul sign rule.\\
We conclude,
\begin{prop}\label{propobstruction}
Let $(C,d,\nabla_C,E^{k,1})$ be a homotopy G-coalgebra. 
\begin{enumerate}
 \item The cobar construction $\Omega C$ is an involutive DG-Hopf algebra if and only if all the obstructions $O_n:\wid{C}\to \wid{C}^{\ot n}$ defined in \eqref{invHcoal2} 
for $n\geq 2$ are zero.
\item Let $C$  be $2$-reduced. If all the obstructions $O_n:\wid{C}\to \wid{C}^{\ot n}$  are zero, then the double cobar construction  $\Omega^2C$ is a homotopy BV-algebra whose the BV-operator is the Connes-Moscovici operator defined in \eqref{BVop}.
\end{enumerate}
\end{prop}
In fact we can make more precise the second point of the previous proposition.
Let $M$ be a DG-module and let $M_{\leq n}$ be the sub-module of elements $m\in M$ of degree $|m|\leq n$.
 The homotopy G-coalgebra structure on a connected coalgebra $C$ preserves the filtration $C_{\leq n}$.
Indeed the degree of the operations $E^{k,1}:\wid{C}\to \wid{C}^{\ot k}\ot \wid{C}$ is $0$. We have,
\begin{prop} 
Let $(C,d,\nabla_C,E^{k,1})$ be an $i$-reduced homotopy G-coalgebra, $i\geq 2$.
If there exists an integer $n$ such that $O_k=0$ for $ki\leq n-1$, then $\Omega^2 (C_{\leq n})$ is a homotopy BV-algebra. 
\end{prop}
 \begin{proof}
The degree of $O_k:\wid{C}\to \wid{C}^{\ot k}$ is $0$. Then on $\ds (\ov{C}_{\leq n})=\wid{C}_{\leq n-1}$ the only eventually non-zero operations $O_k:\wid{C}_{\leq n}\to (\wid{C}_{\leq n})^{\ot k}$ are those with $ki\leq n-1$. 
 \end{proof}
\begin{rmq}
 The two previous propositions work for a Hirsch coalgebra $C$ instead of a homotopy G-coalgebra : the operations from \eqref{invHcoal2} have the same definition in terms of $F_n$; only the $F_n$'s defining the antipode differ. However, the involved techniques above are quite similar for Hirsch coalgebras.
\end{rmq}
\begin{rmq}
 For a general Hirsch coalgebra $(C,E^{i,j})$ the condition of cocommutativity of the coproduct on $\Omega C$ is $E^{i,j}=\tau E^{j,i}$ for all $(i,j)$. Accordingly, with Proposition \ref{antip properties} we see (already in component three) that condition for the antipode to be involutive is weaker than condition for coproduct to be cocommutative. When the Hirsch coalgebra is a homotopy G-coalgebra, the cocommutativity of the coproduct means that the homotopy G-coalgebra is quasi trivial: $E^{1,1}=\tau E^{1,1}$ and $E^{k,1}=0$ for $k\geq 2$. 
\end{rmq}
\section{Applications}
\subsection{On a homotopy G-coalgebra structure of $C_*(X)$}\label{section hGc}
The chain complex $C_*(X)$ of a topological space or a simplicial set $X$ has a rich algebraic structure, it is an $E_{\infty}$-coalgebra. This was treated in many papers including 
J. McClure and J. Smith \cite{McClureSmithmulti}, C. Berger and B. Fresse \cite{BergerFresse}.
For example, the surjection operad $\chi$ introduced in \cite{McClureSmithmulti} acts on the normalized chain complex $C_*(X)$ of a simplicial set $X$ making it into a coalgebra over $\chi$, see \cite{McClureSmithmulti, BergerFresse}. 
It has a filtration of suboperads
\[
 F_1\chi \subset  F_{2}\chi \subset \cdots \subset F_{n-1}\chi \subset F_n\chi\subset \cdots \subset \chi.
\]
This structure leads to a coproduct $\nabla_0: \cob \to \cob\ot \cob$ first defined by Baues in \cite{Bauesgeom}. 
The third stage filtration $F_3\chi$ gives a homotopy cocommutativity $\nabla_1:\cob\to \cob\ot \cob$ to Baues' coproduct, see \cite{Bauesgeom}. In turn the operation $\nabla_1:\cob\to \cob\ot \cob$ is cocommutative up to a homotopy $\nabla_2:\cob\to \cob\ot \cob$ and so on. The resulting structure is known as a structure of DG-bialgebra with Steenrod coproduct $\nabla_i$. This was achieved by Kadeishvili in \cite{KadeishviliSteenrod} where the corresponding operations in $\chi$ are given. 
\\

Baues' coproduct \cite[p.334]{Bauesdouble}, \cite[(2.9) equation (3)]{Bauesgeom} on the cobar construction $\cob$ corresponds to a homotopy G-coalgebra on $C_*(X)$. 
By a direct comparison, we see that this homotopy G-coalgebra structure coincides with the one given in \cite{McClureSmithmulti, BergerFresse}. To be more precise, let 
\[
E^{k,1}_0:\wid{C}_*(X)\to \wid{C}_*(X)^{\ot k}\ot \wid{C}_*(X) 
\]
be the operations defined by Baues' coproduct 
\[
\nabla_0:\cob \to \cob \ot \cob .
\]
 Let 
\[
E^{1,k}:\wid{C}_*(X)\to \wid{C}_*(X)\ot \wid{C}_*(X)^{\ot k} 
\]
the operations defined by
\begin{multline*}
  E^{1,k}(\ds\sigma)=\sum_{0\leq j_1<j_2<...<j_{2k}\leq n} \\
 \pm \ds\sigma(0,...,j_1,j_2,...,j_3,j_4,\cdots,j_{2k-1},j_{2k},...,n)\ot \ds\sigma(j_1,...,j_2)\ot\ds \sigma(j_3,...,j_4)\ot \cdots \ot \ds\sigma(j_{2k-1},...j_{2k}),
 \end{multline*}
for $\ds\sigma \in \wid{C}_n(X)$. 
These are the operations denoted by $AW(1,2,1,3,...,k-1,1,k)$ in \cite[section 2.2]{BergerFresse} where $AW(u):\wid{C}_*(X)\to \wid{C}_*(X)^{\ot n}$ is defined for a surjection $u\in \chi(n)_d$.
The operad $F_2\chi$ is generated by the surjections $(1,2)$ and $(1,2,1,3,...,1,k,1,k+1,1)$ for $k\geq 1$. 
Then the operations
\[
 E^{1,k}:=AW(1,2,1,3,...,k,1,k+1,1),
\]
define a homotopy G-coalgebra structure on $C_*(X)$. 
We have
\begin{prop}\label{hGc defprop} Let $X$ be a $1$-reduced simplicial set. Then 
\[
E_0^{k,1}=\pm\tau^{1,k} E^{1,k}, 
\]
where $\tau^{1,k}$ orders the factors as the following  permutation  
\[
\begin{tikzpicture}
\matrix [matrix of math nodes,left delimiter={(},right delimiter={)}]
{
1 & 2 & \cdots & k & k+1 \\
2 & 3 & \cdots & k+1 & 1 \\
};
\end{tikzpicture}.
\]
\end{prop}
\begin{proof}
First we recall Baues' coproduct $\nabla_0$. We adopt the same conventions as in \cite{Bauesdouble}. 
 For $\sigma_i\in X$, the tensor $[\sigma_1|\sigma_2|...|\sigma_n]$ is the tensor $s^{-1}\sigma_{i_1}\otimes s^{-1}\sigma_{i_2}\otimes ...\otimes s^{-1}\sigma_{i_r}$ where the indices $i_j$ are such that $\sigma_{i_j}\in X_{n_{i_j}}$ with $n_{i_j}\geq 2$. 
For a subset $b=\{b_0<b_1<...<b_r\}\subset\{0,1,...,n\}$ we denote by $i_b$ the unique order-preserving injective function 
\[
i_b:\{0,1,...,r\}\to\{0,1,...,n\}
 \]
such that $\Ima(i_b)=b$. 
Let $\sigma \in X_n$ and $0\leq b_0<b_1\leq n$. 
We denote by $\sigma(b_0,...,b_1)$ the element $i^*_b\sigma\in X_{b_1-b_0}$ where $b=\{b_0,b_0+1,...,b_1-1,b_1\}$.  
Let $b\subset \{1,...,n-1\}$, we denote by $\sigma(0,b,n)$ the element $i^*_{b'}\sigma$ where $b'=\{0\}\cup b \cup \{n\}$.\\
Baues' coproduct $\nabla_0$  \cite[p.334]{Bauesdouble} is defined on $\sigma \in X_n$, $n\geq 2$ by:
\begin{equation}\label{coprod}
 \begin{split}
  \nabla_0 : \cob & \to \cob \otimes \cob \\
      s^{-1}\sigma=[\sigma] &\mapsto \sum_{\substack{b=\{b_1<b_2<...<b_r\} \\b\subset \{1,...,n-1\}}} (-1)^{\zeta} [\sigma(0,...,b_1)|\sigma(b_1,...,b_2)|...|\sigma(b_r,...,n)]\otimes[\sigma(0,b,n)]
 \end{split}
\end{equation} where
\begin{align*}
\zeta&=r|\sigma(0,..,b_1)|+\sum_{i=2}^{r}(r+1-i)|\sigma(b_{i-1},...,b_i)|-r(r+1)/2,
 \end{align*}
it is extended as an algebra morphism on the cobar construction.

We show that the coproduct $\nabla_0$ defines the same homotopy G-coalgebra structure on $C_*(X)$ when $X$ is $1$-reduced.
Let $\sigma$ be an $n$-simplex, then Baues' coproduct  is a sum over the subset  $b=\{b_1<b_2<...<b_r\}\subset \{1,...,n-1\}$. 
For such a $b$, we set $b_0:=0$, $b_{r+1}:=n$ and we defined $\beta_l\in b\cup \emptyset$ as follows.
For $l=1$,
\[
\beta_{1}:=\min_{1\leq i\leq r+1}\{b_i~|~b_i-b_{i-1}\geq 2\}, 
\]
and let  $\eta_1$ be the index such that $b_{\eta_1}=\beta_1$; for $l\geq 2$, 
\[
\beta_{l}:=\min_{\eta_{l-1}+1\leq i\leq r+1}\{b_i~|~b_i-b_{i-1}\geq 2\} 
\]
 where $\eta_l$ is the index such that $b_{\eta_{l}}=\beta_{l}$. 
Moreover, we set $\alpha_l:=b_{\eta_l-1}$. 
Let $k$ be the integer such that $1\leq l \leq k$.
Explicitly, $k$ is given by:
\begin{equation*}
 k=\# \{1\leq i\leq r+1~|~ b_i-b_{i-1}\geq 2\}.
\end{equation*}
Thus the coproduct $\nabla_0$ is 
\begin{multline*}
      \nabla_0([\sigma])=\sum_{k=0}^n \sum_{\substack{0\leq \alpha_1\leq \beta_1\leq \alpha_2 \leq \cdots\leq \alpha_k\leq \beta_k\leq n \\ \beta_i-\alpha_i\geq 2,~ 1\leq i\leq k }} \pm \ds\sigma(\alpha_1,...,\beta_1)\ot \ds \sigma(\alpha_2,...,\beta_2)\ot \cdots\\
\cdots \ot \ds \sigma(\alpha_k,...,\beta_k)\otimes
 \ds \sigma(0,...,\alpha_1,\beta_1,...,\alpha_2,\cdots,\alpha_k,\beta_k,...,n).
\end{multline*}
Thus we have,
\begin{multline*}
 E^{k,1}(\ds\sigma)=\sum \pm\ds\sigma(\alpha_1,...,\beta_1)\ot \ds\sigma(\alpha_2,...,\beta_2)\ot\cdots\ot \ds\sigma(\alpha_k,...,\beta_k)\ot \\  \ds \sigma(0,...,\alpha_1,\beta_1,...,\alpha_2,\cdots,\alpha_k,\beta_k,...,n).
\end{multline*}
The result is obtained by setting $j_{2l-1}=\alpha_{l}$ and $j_{2l}=\beta_{l}$.
Since $X$ is $1$-reduced, the elements $\sigma(j_l,...,j_{l+1})$ such that $j_{l+1}-j_l=1$ are elements in $X_1=*=s_0(*)$ and then are degenerate.
\end{proof}

In the two next subsections \ref{Section4} and \ref{Section5} we adopt the above notations for the homotopy G-coalgebra structure $E^{1,k}$ on $C_*(X)$.%
\subsection{Obstruction to the involutivity of the antipode of $\Omega C_*(\Sigma X)$}\label{Section4}
Let  $\Sigma X$ be a simplicial suspension of a simplicial set $X$. We show that the family of obstructions $O_n:\wid{C}_*(\Sigma X)\to \wid{C}_*(\Sigma X)^{\ot n}$  defined in \eqref{invHcoal2}
can be reduced to $O_2:\wid{C}_*(\Sigma X)\to \wid{C}_*(\Sigma X)^{\ot 2}$; the latter being governed by 
the (lack of) cocommutativity of the  operation $E^{1,1}$. 
\begin{de}{\cite[Definition 27.6 p.124]{MaySimplicial}}
Let $X$  be a simplicial set such that $X_0=*$ and with face and degeneracy operators $s_j:X_n\to X_{n+1}$ and $d_j:X_n\to X_{n-1}$. 
 The \textbf{simplicial suspension} $\Sigma X$ is defined as follow.
 The component $(\Sigma X)_0$ is just an element $a_0$ and $(\Sigma X)_n=\{(i, x)\in \mathbb{N}^{\geq 1} \times X_{n-i} \}/ \Bigl((i,s_0^n(*)) = s_0^{n+i}(a_0)\Bigr)$. We set $a_n=s_0^n(a_0)$. The face and degeneracy operators are generated by:
\begin{itemize}
 \item $d_0(1, x) = a_n$ for all $x\in X_n$;
\item  $d_1(1, x) = a_0$ for all $x \in X_0$;
\item $d_{i+1}(1, x) = (1, d_i(x))$ for all $x \in X_n, n>0$;
\item $s_0(i, x) = (i + 1, x)$;
\item $s_{i+1}(1, x) = (1, s_i(x))$,
\end{itemize}
with the other face and degeneracy operators determined by the requirement that $\Sigma X$ is a simplicial set.
\end{de}
\begin{prop}{\cite{HPS}}
The differential of $(\Omega C_*(\Sigma X),d_0+d_1)$ is reduced to its linear part $d_0$. 
\end{prop}
\begin{proof}
 The only non degenerate elements in $\Sigma X$ are $(1,x)$ and we have that $d_0(1, x) = a_n$ is  degenerate. Therefore the Alexander-Whitney coproduct on $C_*(\Sigma X)$ is primitive (we recall that $C_*$ are the normalized chains). Then, the reduced coalgebra $\ov{C}_*(\Sigma X)$ is a trivial coalgebra; 
so the quadratic part $d_1$ of the  differential $d_{\Omega}=d_0+d_1$ on the cobar construction is trivial.
\end{proof}

A natural coproduct to define is the shuffle coproduct (primitive on cogenerators) which gives a cocommutative DG-Hopf structure to $\Omega C_*( \Sigma X)$; applying Proposition \ref{PropBV}  we obtain on $\Omega^2 C_*(\Sigma X)$ a homotopy BV-algebra structure. However, this coproduct does not correspond to Baues' coproduct $\nabla_0$. Indeed, the  homotopy G-algebra structure on $C_*(\Sigma X)$ is not completely trivial. 
Because of the cocommutativity of the Alexander-Whitney coproduct on $C_*(\Sigma X)$, the operation $E^{1,1}$ must be a chain map but not necessary the zero map. 
We obtain,
\begin{prop}\label{susp1} The homotopy G-coalgebra structure $E^{k,1}$ on $C_*(\Sigma X)$ given by Baues' coproduct is  
 \begin{equation*}
 E^{1,1}(\ds\sigma)= \sum_{1<l< n}\pm \ds\sigma(0,...,l)\ot \ds\sigma(0,l,l+1,...,n);
\end{equation*}
\begin{equation*}
 E^{k,1}=0 ~~\text{for}~k\geq 2.
\end{equation*}
\end{prop}
\begin{proof}
 For a non-degenerate $\sigma \in (\Sigma X)_n$, we have $d_0\sigma=\sigma(1,...,n)$ is degenerate. Consequently, the operation
\begin{align*}
 E^{1,1}(\ds\sigma)= \sum_{k<l}\pm \ds\sigma(k,...,l)\ot \ds\sigma(0,...,k,l,l+1,...,n)
\end{align*}
reduces to 
 \begin{equation*}
 E^{1,1}(\ds\sigma)= \sum_{1<l< n} \pm\ds\sigma(0,...,l)\ot \ds\sigma(0,l,l+1,...,n).
\end{equation*}
The higher operations $E^{k,1}$ for $k\geq 2$ given by
\begin{multline*}
 E^{1,k}(\ds\sigma)=\sum_{0\leq j_1<j_2<...<j_{2k}\leq n} \\
\pm\ds\sigma(0,...,j_1,j_2,...,j_3,j_4,\cdots,j_{2k-1},j_{2k},...,n)\ot\ds \sigma(j_1,...,j_2)\ot\ds\sigma(j_3,...,j_4) \ot \cdots \ot \ds\sigma(j_{2k-1},...j_{2k})
\end{multline*}
are zero since the terms $\sigma(j_3,...,j_{4})$ are degenerate. 
\end{proof}
The triviality of higher operations $E^{k,1}$ for $k\geq 2$ does not imply the vanishing of higher obstructions $O_n$ since the $E^{1,1}$ operation appears in all the $O_n$'s. However, we can reduce the family of obstructions $O_n$ to only $O_2$ in this case.
 \begin{prop}\label{obstructionSuspension1}
  Let $(C,d,\nabla_C,E^{k,1})$ be a homotopy G-coalgebra with $E^{k,1}=0$ for $k\geq 2$. If $O_2=E^{1,1}-\tau E^{1,1}$ is zero, then so is $O_n$ for $n\geq 2$.
 \end{prop}
\begin{proof}
 By making explicit the coassociativity of the coproduct $\nabla:\Omega C \to \Omega C$  we obtain, in particular, the equation   
\begin{equation*}
 (1\ot E^{1,1})E^{1,1}- (E^{1,1}\ot 1)E^{1,1} = (\tau\ot 1)E^{2,1} + (1\ot 1)E^{2,1}.
\end{equation*} 
Therefore the triviality of the higher operation $E^{2,1}$ implies that $E^{1,1}$ is coassociative. Together with the vanishing of $O_2$ we obtain a coassociative and cocommutative operation $E^{1,1}$. 
We use this fact to vanish the $O_n$'s :  by \eqref{Fn}, the terms $F^i$, $i\leq n$ involved in $O_n$ are 
\begin{align*}
F^{n}=&(-1)^{n}(1^{\ot n-2}\ot E^{1,1})\cdots (1\ot E^{1,1})E^{1,1}.
\end{align*} 
Now because of the coassociativity of $E^{1,1}$ we can write each term $\mu_{\Omega_{(12)}}^{(s)}(F^{n_1}\ot ...\ot F^{n_s})F^s$ of the sum
\begin{align*}
O_n&=\sum_{s=1}^{n}\sum_{n_1+...+n_s=n}\mu_{\Omega_{(12)}}^{(s)}(F^{n_1}\ot ...\ot F^{n_s})F^s
\end{align*}
as $\pm \sigma\circ F^n$  where $\sigma$ is a permutation (depending of the term we consider). Now using the cocommutativity and again the coassociativity of $E^{1,1}$ we can remove all transpositions to obtain $\pm F^n$. A direct counting shows that positive and negative terms are equal in number :
the sign of  $(F^{n_1}\ot ...\ot F^{n_s})F^s$ is $(-1)^{n_1+...+n_s+s}=(-1)^{n+s}$ and the number of partitions of $n$ into $s$ integers $\geq 1$ is $\binom{n-1}{s-1}$.
\end{proof}
 Therefore we can consider $O_2$ as the only obstruction to the involutivity of the  antipode on the cobar construction. 

\subsection{Homotopy BV-algebra structure on $\Omega^2 C_*(\Sigma^2 X)$}\label{Section5}
Here we prove that the homotopy G-coalgebra structure on $C_*(\Sigma^2 X)$ is trivial and then that $\Omega C_*(\Sigma^2 X)$ is an involutive DG-Hopf algebra with the shuffle coproduct.
\begin{theo}\label{DoubleSuspension}
 Let $\Sigma^2 X$ be a  double suspension. Then :
\begin{itemize}
 \item the homotopy G-coalgebra structure on $C_*(\Sigma^2 X)$ corresponding to Baues' coproduct 
\[
\nabla_0:\Omega C_*(\Sigma^2 X) \to \Omega C_*(\Sigma^2 X)\ot  \Omega C_*(\Sigma^2 X), 
\]
 has trivial higher operations i.e. $E^{1,k}=0$ for $k\geq 1$;
 \item the double cobar construction $\Omega^2 C_*(\Sigma^2 X)$ is a homotopy BV-algebra with the BV-operator from \eqref{BVop}.
\end{itemize}
\end{theo}
\begin{proof}
By Proposition \ref{susp1}  the operations $E^{k,1}$ for $k\geq 2$ are trivial. Therefore to prove the first statement it remains to prove that $E^{1,1}$ is trivial. This follows that  $d_1(1,(1,x))=(1,d_0(1,x))=(1,a_n)=s_{n+1}(1,a_0)$ is degenerate.
Indeed, the operation 
 \begin{equation*}
 E^{1,1}(\ds\sigma)= \sum_{1<l< n} \pm\ds\sigma(0,...,l)\ot\ds \sigma(0,l,l+1,...,n)
\end{equation*}
is trivial since $l>1$.
\\
For the second statement since the homotopy G-coalgebra structure is trivial, all the obstructions $O_n$ are zero and so, by Proposition \ref{propobstruction}, we obtain the  announced homotopy BV-algebra structure on $\Omega^2 C_*(\Sigma^2 X)$. 
\end{proof}
We make explicit the homotopy BV-algebra structure on $\Omega^2 C_*(\Sigma^2 X)$. Let us first observe that Baues' coproduct $\nabla_0$ is a shuffle coproduct.
Indeed, the homotopy G-coalgebra structure on $C_*(\Sigma^2 X)$ being trivial, for any element $a\in \ov{C}_*(\Sigma^2 X)$ we have 
\begin{equation*}
 \nabla_0(\ds a)= (E_{1,0} + E_{0,1})(\ds a)=   \ds a \ot 1+ 1\ot \ds a\in \Omega C_*(\Sigma^2 X)\ot \Omega C_*(\Sigma^2 X).
\end{equation*}
Therefore its extension $\nabla_0:\Omega C_*(\Sigma^2 X)\to \Omega C_*(\Sigma^2 X)\ot \Omega C_*(\Sigma^2 X)$ as algebra morphism is the shuffle coproduct, see \cite[Theorem III.2.4]{Kassel}.
We write $\nabla_0$ as 
\begin{align*}
 \nabla_{0}([a_1|...|a_n])=\sum_{I_0\cup I_1}\pm [a_{I_0}]\ot [a_{I_1}]
\end{align*}
where the sum is taken over all partitions $I_0\sqcup I_1$ of $I=\{1,...,n\}$ with $I_0=\{i_1<i_2<...<i_k\}$ and $I_1=\{j_1<j_2<...<j_{n-k}\}$. We denote $a_{I_0}$ to be $a_{i_1}\ot a_{i_2}\ot...\ot a_{i_k}$. Also we denote $a_{I_0^{-1}}$ to be $a_{i_k}\ot ...\ot a_{i_2}\ot a_{i_1}$ and similarly for $a_{I_1}$.\\

Now we make explicit the BV-operator on a $2$ and $3$-tensor. 
For the sake of simplicity we do not keep track of signs; also to avoid confusion we write $\uv{a}=[a_{1}|a_{2}|...|a_{k}]$ for the basic  elements of $\Omega C_*(\Sigma^2 X)$ and $[\uv{a}_{1},\uv{a}_{2},...,\uv{a}_{n}]$ for the basic elements of $\Omega^2 C_*(\Sigma^2 X)$.\\
Thus $[\uv{a}_{1},\uv{a}_{2},...,\uv{a}_{n}]$ is $[[a_{1,1}|a_{1,2}|...|a_{1,k_{1}}],[a_{2,1}|a_{2,2}|...|a_{2,k_{2}}],...,[a_{n,1}|a_{n,2}|...|a_{n,k_{n}}]]$.\\
With these conventions, the BV-operator $\Delta$ is given by 
\begin{align*}
 \Delta([\uv{a},\uv{b}])=\Delta([[a_1|...|a_m],[b_1|...|b_n]])&=[[a_m|...|a_1|b_1|...|b_n]] + [[b_n|...|b_1|a_1|...|a_m]]. \\
 \Delta([[a_1|...|a_m],[b_1|...|b_n],[c_1|...|c_r]])&=[[a_{I_1^{-1}}|b_1|...|b_n],[a_{I_0^{-1}}|c_1|...|c_r]]\\
& +[[b_{I_1^{-1}}|c_1|...|c_r],[ b_{I_0^{-1}}|a_1|...|a_m]] \\
&+[[c_{I_1^{-1}}|a_1|...|a_m],[c_{I_0^{-1}}|b_1|...|b_n]].
\end{align*}
The homotopy $H$ from Proposition \ref{PropBV} is given by
\begin{align*}
 H([\uv{a},\uv{b}])&=0.
 \end{align*}
To know explicitly each term of 
\[
\Delta([\uv{a},\uv{b},\uv{c}])-[\uv{a},\uv{b},\Delta([\uv{c}])]-[\Delta([\uv{a},\uv{b}]),\uv{c}]-\{[\uv{a},\uv{b}] ;[\uv{c}] \} =\partial  H([\uv{a},\uv{b}];[\uv{c}])
\]
we need to know the following terms:
 \begin{align*}
 H([\uv{a},\uv{b}];[\uv{c}])=H([[a_1|...|a_m],[b_1|...|b_n]];[[c_1|...|c_r]])&=\pm [[b_n|...|b_1|a_1|...|a_m|c_1|...|c_r]],
\end{align*}
 \begin{align*}
 H([\uv{a},\uv{b},\uv{c}];[\uv{d}])&=H([[a_1|...|a_m],[b_1|...|b_n],[c_1|...|c_r]];[[d_1|...|d_s]])\\
&= \pm [[c_{I_1^{-1}}|a_1|...|a_m|d_1|...|d_s],[c_{I_0^{-1}}|b_1|...|b_n]]\\
& \pm [[c_{I_1^{-1}}|a_1|...|a_m],[c_{I_0^{-1}}|b_1|...|b_n|d_1|...|d_s]]\\
& \pm [[b_{I_1^{-1}}|c_1|...|c_r],[b_{I_0^{-1}}|c_1|...|c_r|d_1|...|d_s]],
\end{align*}
and
\begin{align*}
 H([\uv{a},\uv{b}];[\uv{c},\uv{d}])&=\pm [[c_{I_1^{-1}}|a_1|...|a_m|d_1|...|d_s],[c_{I_0^{-1}}|b_1|...|b_n]].
\end{align*}

\begin{rmq}
 From a topological point of view, let us consider $X$ to be a connected (countable) CW-complex with one vertex. The James/Milgram's models $J_i(X)$ are H-spaces homotopically equivalent to $\Omega^i\Sigma^i X$, see \cite[theorem 5.2]{Milgram}. Moreover, for $i\geq 1$ the cellular chain complex $C_*(J_{i}(\Sigma X))$ is a cocommutative, primitively generated DG-Hopf algebra,  and $C_*(J_{i+1}(X))$ is isomorphic to $\Omega C_*(J_i(\Sigma X))$, \cite[theorem 6.1 and 6.2]{Milgram}. Then using Proposition \ref{PropBV} we get a homotopy BV-algebra structure on it which is similar to the one obtained in the simplicial context. Also, we have a DGA-quasi-isomorphism 
\begin{center}
\begin{tikzpicture} [>=stealth,thick,draw=black!50, arrow/.style={->,shorten >=1pt}, point/.style={coordinate}, pointille/.style={draw=red, top color=white, bottom color=red}]
\matrix[row sep=9mm,column sep=16mm,ampersand replacement=\&]
{
 \node (b) {$\Omega C_*(J_1(\Sigma X))$}; \& \node (c){$C_*(J_2(X))$};\& \node (d){$C_*(\Omega^2\Sigma^2 X).$};\\
}; 
\path
         (b)     edge[arrow,->]      node[label=above:{\scriptsize}] {} (c)
	 (c)     edge[arrow,->]      node[label=above:{\scriptsize$C_*(j_2)$}] {} (d)
        ; 
\end{tikzpicture}
\end{center}
\end{rmq}

\subsection{Homotopy BV-algebra structure on $\ccob$ over $R\supset\mathbb{Q}$}\label{sectionQ}

In \cite{Bauesgeom}, Baues gives an explicit construction of the cobar construction of a $1$-reduced simplicial set $X$ over the integer coefficient ring as a cocommutative up to homotopy DG-bialgebra. By using methods of \cite{Anick}, he shows that over a ring $R$ containing $\mathbb{Q}$ as a subring, this DG-bialgebra can be deformed into a strictly cocommutative DG-bialgebra \cite[Theorem 4.7]{Bauesgeom}. The latter is isomorphic as DG-Hopf algebras to the universal enveloping algebra of the Lie algebra $L(\ds\ov{C}_*X)$ generated by the desuspension of the reduced coalgebra of normalized chain complex, combine \cite[Theorem 4.8]{Bauesgeom} and \cite[Proposition V.2.4]{Kassel}. Hence, over such a ring $R$, the cobar construction $\cob$ is (can be deformed into) an involutive DG-Hopf algebra.
By  applying Proposition \ref{PropBV}  to the resulting involutive cobar construction $(\cob,\nabla'_0,S')$, we obtain 
\begin{theo}\label{CorBV}
 Let $X$ be a $2$-reduced simplicial set. Then the double cobar construction $\Omega (\cob,\nabla'_0,S')$ over $R\supset\mathbb{Q}$ coefficients is a homotopy BV-algebra  with BV-operator the Connes-Moscovici operator.
\end{theo}

In the sequel we detail how the antipode is deformed.
We extend the deformation (given in \cite{Bauesgeom}) of the DG-bialgebra structure of the cobar construction to a deformation of the DG-Hopf structure. The resulting cobar construction $(\cob,\nabla'_0,S')$ comes with (anti)derivation homotopies  (see below) connecting the obtained DG-Hopf structure $(\nabla'_0,S')$ with the initial one $(\nabla_0,S)$.\\

First of all we extend some definitions from \cite{Bauesgeom}.
\newcommand{\cpr}{\overline{\nabla}}
\\

Let \textbf{DGA}$_0$ be the category of connected DG-algebras (associative). For a free DG-module $V$, we denote by $L(V)$ the DG-Lie algebra which is the  free graded Lie algebra on the free module $V$.
We set $V_{\leq n}$ (resp. $V_{< n}$)  the sub DG-module of elements $v\in V$ such that $|v|\leq n$ (resp. $|v|<n$).
 Similarly, $V_{\geq n}$ (resp. $V_{> n}$)  denotes the subset of elements $v\in V$ such that $|v|\geq n$ (resp. $|v|>n$).
\begin{de}
Let $f,g:A\to B$ be two maps between two DG-modules $A,B$. A \textbf{derivation homotopy} between $f$ and $g$ is a map $F:A\to B$ satisfying 
\begin{align}
dF+Fd&=f-g \\
F(ab)&=F(a)g(b)+(-1)^{|F||a|}f(a)F(b)~~ \forall a,b\in A. 
\end{align}
An \textbf{antiderivation homotopy} between $f$ and $g$ is a map $\Gamma:A\to B$ satisfying
\begin{align}
d\Gamma+\Gamma d&=f-g \\
\Gamma(ab)&=(-1)^{|a||b|}(\Gamma(b)g(a)+(-1)^{|\Gamma||a|}f(b)\Gamma(a))~~ \forall a,b\in A.
\end{align} 
\end{de}
\begin{de}
 A \textbf{homotopy DG-bialgebra} $(A,\nabla,G_1,G_2)$ is an object $A$ in $\textbf{DGA}_0$ together with a coproduct $\nabla$ in $\textbf{DGA}_0$ making $(A,\nabla)$ into a coalgebra in $\textbf{DGA}_0$ cocommutative up to $G_1$, coassociative up to $G_2$, both being derivation homotopies.
\end{de}
\begin{de}
Let $(A,\nabla,G_1,G_2)$ be a homotopy DG-bialgebra such that $A=TV$, $V$ being a DG-module with $V_0=0$. The counit is the augmentation of $TV$.
We call $(A,\nabla,G_1,G_2,V)$  \textbf{$n$-good} if the following conditions \eqref{ngood1}, \eqref{ngood2},\eqref{ngood3} hold.
\begin{align}
\nabla=\tau \nabla ~\text{and}~ (\nabla\ot1)\nabla=(1\ot \nabla)\nabla ~\text{on} ~ V_{\leq n}  \label{ngood1}\\
d(V_{\leq n+1})\subset L(V_{\leq n}) \subset T(V)=A \label{ngood2}\\
V_{\leq n} \subset \ker (\overline{\nabla}).\label{ngood3}
\end{align}
If moreover there exists a chain map $S:A\to A$ such that the condition: 
\begin{equation}
\mu(1\ot S)\nabla =\eta\epsilon=\mu(S\ot 1)\nabla ~~\text{on}~~ V_{\leq n}
\end{equation}
is satisfied, then $A$ is called \textbf{$n$-good$^+$}. 
\end{de}

\begin{lem}\label{googHopf}
 Let $\mathcal{A}:=(A=TV,\nabla,G_1,0,S, V)$ be a homotopy DG-bialgebra with antipode $S$. We suppose it is $n$-good$^+$ for a map $S:A\to A$. 
Then there exists a homotopy DG-bialgebra $(A,\nabla^{n+1},G_1^{n+1},G_2^{n+1},S^{n+1},V^{n+1})$ which both extends $\mathcal{A}$ and is $(n+1)$-good$^+$.
Moreover, there is a derivation homotopy $F^{n+1}:\nabla\simeq \nabla^{n+1}$ and an antiderivation homotopy $\Gamma^{n+1}:S\simeq S^{n+1}$ with $F^{n+1}(a)=0$ and $\Gamma^{n+1}(a)=0$ for $|a|<n$, $a\in A$.
\end{lem}
\begin{proof}
The part $(n+1)$-good as bialgebra is already done in \cite[Theorem 4.5]{Bauesgeom} where $\nabla^{n+1},G_1^{n+1}$,$G_2^{n+1},F^{n+1}$ and $V^{n+1}$ are defined. 
Therefore we only need to construct $S^{n+1}$ and $\Gamma^{n+1}$.\\
We recall that $(V^{n+1})_k=V_k$ for $k\neq n+1,n+2$.
Let $S^{n+1}=S-R^{n+1}$ where $R^{n+1}=\mu(1\ot S)\nabla^{n+1} -\eta\epsilon$. Then $R^{n+1}=0$ on $V_{\leq n}$.
On $(V^{n+1})_{n+1}$ we have
\begin{align*}
  \mu(1\ot S^{n+1})\nabla^{n+1} &= \mu(1\ot S)\nabla^{n+1} - \mu(1\ot R^{n+1})\nabla^{n+1} \\
&= \mu(1\ot S)\nabla^{n+1} - R^{n+1}  \\
&= \mu(1\ot S)\nabla^{n+1} - \mu(1\ot S)\nabla^{n+1} +\eta\epsilon \\
&= \eta\epsilon.
\end{align*}
The homotopy $F^{n+1}$ between $\nabla$ and $\nabla^{n+1}$ is such that $F^{n+1}=0$ on $(V^{n+1})_{<n}\cup (V^{n+1})_{>n+1}$.
Setting $\Gamma^{n+1}=-\mu(1\ot S)F^{n+1}$ we obtain the desired homotopy. Indeed, 
\begin{align*}
d\Gamma^{n+1}+\Gamma^{n+1}d=-\mu(1\ot S)(dF^{n+1}+F^{n+1}d)=\mu(1\ot S)(\nabla^{n+1}-\nabla) =R^{n+1}=S-S^{n+1}. 
\end{align*}
And $\Gamma^{n+1}$ is such that $\Gamma^{n+1}=0$ on $(V^{n+1})_{<n}\cup (V^{n+1})_{>n+1}$.\\
Next we extend $S^{n+1}$ on $T((V^{n+1})_{\leq n+1})$ as an algebra antimorphism and we extend  $\Gamma^{n+1}$ as an antiderivation homotopy.
We have $\mu(S^{n+1}\ot 1)\nabla^{n+1}=\eta\epsilon$ on $(V^{n+1})_{\leq n+1}$ since $\nabla^{n+1}$ is coassociative on $(V^{n+1})_{\leq n+1}$.
\end{proof}

\begin{lem}\label{ngoodlr}
Let $(A=TV,\nabla,G_1,S)$ be a DG-Hopf algebra cocommutative up to the homotopy $G_1$ and with $S$ as antipode. Suppose that $(A,\nabla,G_1,S,V)$ is $1$-good$^+$.
Then there is a coproduct $\nabla'$ and an antipode $S'$ such that $(A,\nabla',S')$ is a cocommutative DG-Hopf algebra.  Moreover there is a  derivation homotopy $F:\nabla\simeq \nabla'$ and an antiderivation homotopy $\Gamma:S\simeq S'$.
\end{lem}
\begin{proof}
An iteration of Lemma \ref{googHopf} yields, for each $n\geq 1$, an $n$-good$^+$ homotopy DG-bialgebra $(A,\nabla^{n+1},G_1^{n+1},G_2^{n+1},S^{n+1},V^{n+1})$. 
Hence we define $\nabla'(v)=\nabla^n(v)$ and $S'(v)=S^n(v)$ for $|v|=n$.
We have $\mu(1\ot S')\nabla'=\eta\epsilon$.
The derivation homotopy  $F$ is defined as $\sum_{n\geq 1}F^{n}$ and 
the antiderivation homotopy $\Gamma$ is defined as $\sum_{n\geq 1}\Gamma^{n}$.
\end{proof}

\begin{prop}\label{HopfcobC}
 Let $\mathbb{Q}\subset R$ and let $X$ be a $1$-reduced simplicial set. Then there is both a coproduct $\nabla'$ and an antipode $S'$ on the cobar construction $\Omega C_*X$ such that $(\cob,\nabla',S')$ is a cocommutative DG-Hopf algebra. Moreover there is a derivation homotopy $F:\nabla_0\simeq \nabla'$ and an antiderivation homotopy $\Gamma:S_0\simeq S'$, where $\nabla_0$ and $S_0$ are respectively Baues' coproduct \cite[(3)]{Bauesgeom} and the associated antipode. 
\end{prop}
\begin{proof}
 We denote by $G_1$ the homotopy to the cocommutativity of the coproduct $\nabla_0$. It is defined in \cite[(4)]{Bauesgeom}. We apply Lemma \ref{ngoodlr} to $(\cob,\nabla_{0},G_1,0,S,\ds\ov{C}_*X)$ which is $1$-good$^+$.
\end{proof}

\section*{Appendix}
\addcontentsline{toc}{section}{Appendix}
Here we recall and develop some facts about the Hirsch and the homotopy G-algebras. 
A Hirsch $(\La,d_{\La},\cdot,\{E_{1,k}\}_{k\geq 1})$ corresponds to a product $\mu_E:\Ba \La \ot \Ba \La \to \Ba \La$ such that $(\Ba \La,d_{\Ba \La}, \mu_E)$ is a unital DG-bialgebra.\\
We write down the relations among the $E_{i,j}$, $i,j\geq 1$ coming from both the associativity of $\mu_E$ and the Leibniz relation $d_{\Ba \La}\mu_E = \mu_E(d_{\Ba \La}\ot 1 + 1\ot d_{\Ba \La})$.\\
We detail the construction of the operation $E_{1,1}$ defined in \eqref{E11}.\\
\subsubsection*{Unit condition}  
For all  $\underline{\s a}=\s a_1\ot ...\ot \s a_i\in B\La$ we have: 
\begin{equation}
\mu_{E}( 1_{\La}\ot \underline{\s a})=\mu_{E}(\underline{\s a}\ot 1_{\La})=\underline{\s a} 
\end{equation}
The product being determined by its projection on $\La$ we have:
\begin{align*}
 pr\mu_E (1_{\La}\ot \underline{\s a})=\wid{E}_{0,i}( 1_{\La}\ot \underline{\s a})=pr (\s a_1\ot ...\ot \s a_i)=\begin{cases}
                                                                                                               a_1 & \text{if} ~~i=1;\\
													      0 &  \text{if} ~~i\neq 1,
                                                                                                              \end{cases}
\end{align*}
and also the symmetric relation.
Thus,
\begin{align}
 E_{0,i}=E_{i,0}=0 ~~\text{for all}~~ & i\neq 1 & \text{and}~~& & E_{0,1}=E_{1,0}=Id_{\La}.
\end{align}
\subsubsection*{Associativity condition} 
With the sign convention \eqref{conv}, the associativity of $\mu_E$ gives on $\La^{\ot i}\ot \La^{\ot j}\ot \La^{\ot k}$:
\begin{multline}\label{assocGerst}
E_{1,k}(E_{i,j}(a_1,...,a_{i};b_1,...b_{j});c_1,...,c_k)+\\
 \sum_{n=1}^{i+j} \sum_{\substack{ 0\leq i_1 \leq ... \leq i_n \leq i \\ 0\leq j_1 \leq ... \leq j_n \leq j  }} 
(-1)^{\alpha_1} E_{n+1,k}(E_{i_1,j_1}(a_1,...,a_{i_1};b_1,...b_{j_1}),...\\...,E_{i-i_n,j-j_n}(a_{i_n+1},...,a_{i};b_{j_n+1},...b_{j});c_1,...,c_k)\\
= \sum_{m=1}^{j+k} \sum_{\substack{ 0\leq j_1 \leq ... \leq j_m \leq j \\ 0\leq k_1 \leq ... \leq k_m \leq k  }} 
(-1)^{\alpha_2} E_{i,m+1}(a_1,...,a_i;E_{j_1,k_1}(b_1,...,b_{j_1};c_1,...c_{k_1}),...\\...,E_{j-j_m,k-k_m}(b_{j_m+1},...,b_{j};c_{k_m+1},...c_{k}))\\
+ E_{i,1}(a_1,...,a_i;E_{j,k}(b_1,...,b_{j};c_1,...c_{k}))
\end{multline}
where 
\begin{align*}
 \alpha_1 =&\sum_{u=1}^{n} \left(\sum_{s=i_{u}+1}^{i_{u+1}}|a_{s}|+i_{u+1}-i_{u}\right)\left( \sum_{s=1}^{j_{u}}|b_{s}|+j_{u}\right)
\end{align*}
\begin{align*}
 \alpha_2 = &\sum_{u=1}^{m} \left(\sum_{s=j_{u}+1}^{j_{u+1}}|b_{s}|+j_{u+1}-j_{u}\right)\left( \sum_{s=1}^{k_{u}}|c_{s}|+k_{u}\right)
\end{align*}
with
\begin{align*}
i_0=0; ~~~j_0=0; ~~~ k_0=0;\\
i_{n+1}=i;  ~~~j_{n+1}=j_{m+1}=j; ~~~ k_{m+1}=k. \\ 
\end{align*}
For $i=j=k=1$, that is on $\La \ot \La\ot \La$, the relation \eqref{assocGerst} gives:
\begin{multline}\label{assoE}
 E_{1,1}(E_{1,1}(a;b);c)=E_{1,1}(a;E_{1,1}(b;c))+ E_{1,2}(a;b,c)+(-1)^{(|b|-1)(|c|-1)}E_{1,2}(a;c,b)\\
-E_{2,1}(a,b;c) -(-1)^{(|a|-1)(|b|-1)}E_{2,1}(b,a;c).
\end{multline}

\subsubsection*{Leibniz relation}
On $\La^{\ot i}\ot \La^{\ot j}$, the projection of $d_{\Ba}\mu_E= \mu_E(d_{\Ba}\ot 1+ 1\ot d_{\Ba})$ gives:
\begin{multline}\label{dgmap}
d_{\La}E_{i,j}(a_1,...,a_i;b_1,...,b_j) \\
+ \sum_{\substack{ 0\leq i_1 \leq i \\ 0\leq j_1 \leq j  }} (-1)^{\beta_2}E_{i_1,j_1}(a_1,...,a_{i_1};b_1,...,b_{j_1})\cdot E_{i-i_1,j-j_1}(a_{i_1+1},...,a_{i};b_{j_1+1},...,b_{j}) =\\
= \sum_{l=1}^{i} (-1)^{\beta_3}E_{i,j}(a_1,...,d_{\La}(a_l),...,a_i;b_1,...,b_j)\\
+ \sum_{l=1}^{i-1}(-1)^{\beta_4}E_{i-1,j}(a_1,...,a_l\cdot a_{l+1},...,a_i;b_1,...,b_j)\\
+\sum_{l=1}^{j} (-1)^{\beta_5}E_{i,j}(a_1,...,a_i;b_1,...,d_{\La}(b_l),...,b_j) \\
+ \sum_{l=1}^{j-1}(-1)^{\beta_6}E_{i,j-1}(a_1,...,a_i;b_1,...,b_l\cdot b_{l+1},...,b_j)
\end{multline}
where
\begin{align*}
\beta_2=& \sum_{s=1}^{i_1} |a_s| +\sum_{s=1}^{j_1} |b_s| + i_1 +j_1 
+ \left(\sum_{s=i_1+1}^{i}|a_s|+i-i_1\right)\left(\sum_{s=1}^{j_1}|b_s|+j_1\right)\\
\beta_3=& \eta_0(\underline{\s a})  \\
\beta_4 =& \eta_1(\underline{\s a}) \\
\beta_5=& \sum_{s=1}^{i} |a_s| + i +  \eta_0(\underline{\s b}) \\
\beta_6 =&  \sum_{s=1}^{i} |a_s| + i+ \eta_1(\underline{\s b}).
\end{align*}
For $\underline{\s a}:=\s a_1\ot \cdots \ot \s a_i $, the following signs  
\begin{align*}
\eta_0(\underline{\s a})&= \sum_{s=1}^{l-1} |a_s|+ l  \\
\eta_1(\underline{\s a})&= \sum_{s=1}^{l} |a_s|+ l,
\end{align*}
are the signs of the differential of the bar construction.
And similarly for $\underline{\s b}=\s b_1\ot \cdots \ot \s b_j$.
\begin{rmq}
When the twisting cochain $\wid{E}:\Ba \La  \ot \Ba \La  \to \La $  has $E_{i,j}=0$ except $E_{1,0}$ and $E_{0,1}$, then $\La$ is a commutative DG-algebra.
\end{rmq}
We recall (cf. Definition \ref{def: homotopy G}) that when the twisting cochain $\wid{E}:\Ba \La  \ot \Ba \La  \to \La $ satisfies $\wid{E}_{i,j}=0$ for $i\geq 2$, then $(\La,d,\mu_{\La},\{E_{1,j}\})$ is called a homotopy G-algebra.
\begin{rmq}\label{rmq: ideal bar}{\cite[Proposition 3.2]{LodayRonco}}
The condition $E_{i,j}=0$ for $i\geq 2$ is equivalent to the following condition:  for each integer $r$,  $I_r:=\oplus_{n\geq r} (\ds\ov{\La})^{\ot n}$ is a right ideal for the product $\mu_{E}:\Ba\La \ot \Ba\La \to \Ba\La$.
\end{rmq}
For a homotopy G-algebra $\La$, the equation \eqref{dgmap} gives the three following equalities.\\
\\
On $\La^{\ot 1}\ot \La^{\ot 1}$:
\begin{align}\label{diff11}
 d_{\La}E_{1,1}(a;b)-E_{1,1}(d_{\La}a;b)+(-1)^{|a|}E_{1,1}(a;d_{\La}b)=(-1)^{|a|}\Bigl(a\cdot b -(-1)^{|a||b|}b\cdot a\Bigr).
\end{align}
On  $\La^{\ot 2}\ot \La^{\ot 1}$:
\begin{align}\label{der11}
  E_{1,1}(a_1\cdot a_2;b)=a_1\cdot E_{1,1}(a_2;b)+(-1)^{|a_2|(|b|-1)}E_{1,1}(a_1;b)\cdot a_2.
\end{align}
On $\La^{\ot 1}\ot \La^{\ot 2}$:
\begin{equation}\label{der21}
\begin{split}
  d_{\La}E_{1,2}(a;b_1,b_2)+E_{1,2}(d_{\La}a;b_1,b_2) +(-1)^{|a|+1} E_{1,2}(a;d_{\La}b_1,b_2)+ (-1)^{|a|+|b_1|} E_{1,2}(a;b_1,d_{\La}b_2)\\ =(-1)^{|a|+|b_1|+1} \Bigl( E_{1,1}(a;b_1)b_2 + (-1)^{(|a|-1)|b_1|} b_1E_{1,1}(a;b_2)-E_{1,1}(a;b_1b_2)\Bigr).
\end{split}
\end{equation}

\subsection*{The sign in \eqref{E11}}
Now we give the construction of the operation $E_{1,1}$ in \eqref{E11}.
Recall that $B$ is a DG-bialgebra and that $E_{1,1}:\La \ot \La \to \La$ with $\La:=\Omega B$.\\
First,  we set $E_{1,1}(\ds a;\ds b):=\ds(a\cdot b)$ for all $a,b\in B$. Next, using the equation \eqref{der11}, we extend this to:
\begin{align*}
 E_{1,1}(\ds a_1\ot ... \ot \ds a_m ; \ds b):= \sum_{l=1}^{m} (-1)^{\gamma_1}\ds a_1\ot ...\ot \ds(a_l\cdot b)\ot \ds a_{l+1}\ot... \ot \ds a_m,
\end{align*}
for all homogeneous elements $a_1,a_2,\cdots, a_m,b\in B$, where
\begin{align*}
 \gamma_1 &=|b|\left(\sum_{s=l+1}^{m}|a_s|-m+l\right).
\end{align*}
On the other hand, using a slight abuse of notation, we set 
\begin{equation*}
E_{1,1}(\ds a;\ds b_1 \ot \ds b_2):=(-1)^{|a^2||b_1|+|a^2|}\ds(a^1\cdot b_1)\ot \ds (a^2\cdot b_2).
\end{equation*}
The abuse of notation comes from the fact that the terms  $(-1)^{|a||b_1|+|a|}\ds(b_1) \ot \ds (a\cdot b_2)$ and $\ds(a\cdot b_1)\ot \ds ( b_2)$ belong to the above terms i.e. the coproduct of $B$ evaluated on the element $a$ is not reduced. 
\\
Using the equation \eqref{dgmap} we extend $E_{1,1}$ (using the same abuse of notation) to:
\begin{align*}
 E_{1,1}(\ds a ; \ds b_1 \ot ... \ot \ds b_n):= (-1)^{\gamma_2+|a|+1}\ds (a^1\cdot b_1)\ot ...\ot \ds(a^n\cdot b_n),
\end{align*}
with
\begin{align*}
 \gamma_2 =\gamma_2(a)= &~\kappa_n(a)+ \sum_{u=2}^{n}(|a^u|-1)\left(\sum_{s=1}^{u-1}|b_s|-u+1 \right)\\
&+ \sum_{s=1}^{n}(|a^s|-1)(2n-2s+1) +\sum_{s=1}^{n-1}(|a^s|+|b_s|)(n-s) ,
\end{align*}
where
\begin{align*}
 \kappa_n(a):=\begin{cases}
              \sum_{1\leq 2s+1\leq n}|a^{2s+1}|     & \text{if $n$ is even} ;\\  
	      \sum_{1\leq 2s\leq n}|a^{2s}|     & \text{if $n$ is odd} .\\
	      \end{cases}
\end{align*}
Finally, we find the operation in \eqref{E11}:
\begin{multline*}
 E_{1,1}(\ds a_1 \ot ...\ot\ds a_m ; \ds b_1 \ot ... \ot \ds b_n)\\
:= \sum_{l=1}^{m} (-1)^{\gamma_3}\ds a_1 \ot... (a_l^1\cdot b_1)\ot ...\ot \ds(a_l^n\cdot b_n)\ot ...\ot \ds a_m
\end{multline*}
with 
\begin{align*}
 \gamma_3 &=\sum_{u=1}^{n}|b_u|\left(\sum_{s=l+u}^{m}|a_s|+m-l-u+1\right)+  
\kappa_n(a_l)+ \sum_{u=2}^{n}(|a_l^u|-1)\left(\sum_{s=1}^{u-1}|b_s|-u+1 \right)\\
&+ \sum_{s=1}^{n}(|a_l^s|-1)(2n-2s+1) +\sum_{s=1}^{n-1}(|a_l^s|+|b_s|)(n-s),
 \end{align*}
where
\begin{align*}
 \kappa_n(a):=\begin{cases}
              \sum_{1\leq 2s+1\leq n}|a^{2s+1}|     & \text{if $n$ is even};\\  
	      \sum_{1\leq 2s\leq n}|a^{2s}|     & \text{if $n$ is odd}.\\
	      \end{cases}
\end{align*}

\begin{de}
An $\infty$-morphism between two homotopy G-algebras, say $\La$ and $\La'$, is a  morphism of unital DG-algebras between the associated bar constructions: 
\[
 f: \Ba \La \to \Ba \La'.
\]
\end{de}
Such a morphism is a collection of maps
\[
f_n:\La^{\ot n} \to \La', ~~~n\geq 1,
\]
of degree $1-n$, satisfying the following relations \eqref{eqmorph hGa1} and \eqref{eqmorph hGa2}, 
\begin{equation}\label{eqmorph hGa1}
 \begin{split}
  \sum_{\substack{1\leq r \leq k+l, ~~0\leq k_i \leq 1\\ k_1+...+k_r=k \\ l_1+...+l_r=l}} \pm f_r(E^{\La}_{k_1,l_1}\ot ... \ot E^{\La}_{k_r,l_r}) 
= \sum_{\substack{1\leq w \leq l, ~~  0\leq v \leq 1 \\ i_1+...+i_v=k \\ j_1+...+j_w=l}}\pm E^{\La'}_{v,w}(f_{i_1}\ot  ...\ot f_{i_v};f_{j_1} \ot ...\ot f_{j_w}),
 \end{split}
\end{equation}
for all $k\geq 1$, $l\geq 1$, and  
\begin{align}\label{eqmorph hGa2}
\partial f_n =  \sum_{j+k+l=n}\pm f_{n-1}(1^{\ot j}\ot \mu^{\La}\ot l^{\ot l}) + 
 \sum_{\substack{ j+k=n}}\pm  \mu^{\La'}(f_{j}\ot f_k),
\end{align}
for all $n\geq 1$, where, $\partial$ is the differential of  $\Hom(\La^{\ot n},\La')$.
\\

Now we show that the homology of a homotopy G-algebra is a Gerstenhaber algebra.
To fix the convention:
 \begin{de}
 A Gerstenhaber algebra $(G,\cdot,\{~,~\})$ graded commutative algebra $(G,\cdot)$ endowed with a degree $1$ bracket,
\[
\{~;~\}:G\ot G \to G 
\]
satisfying the following relations:
\begin{align} 
   \{a,b\} &= -(-1)^{(|a|+1)(|b|+1)} \{b,a\};\label{Antisym}\\
\{a,b\cdot c\} &= \{a,b\}\cdot c + (-1)^{(|a|+1)|b|} b\cdot\{a,c\};\label{Poisson}\\
\{a,\{b,c\}\} &= \{\{a,b\},c\} + (-1)^{(|a|+1)(|b|+1)}\{b,\{a,c\}\}.\label{Jacobi}
\end{align}
\end{de}

\begin{prop}\label{rmq E dieze}
Let $(\La,d_{\La},\cdot,E_{1,k})$ be a homotopy G-algebra.
Then the degree $1$ bracket
\begin{equation*}
\{a;b\}=E_{1,1}(a;b)-(-1)^{(|a|-1)(|b|-1)}E_{1,1}(b;a) 
\end{equation*}
defines a Gerstenhaber algebra structure on the homology $H(\La,d_{\La})$.
\end{prop}
\begin{proof}
The equality \eqref{diff11} shows the commutativity of the product of $\La$ up to homotopy.
Indeed, it suffices to set $E^\#_{1,1}(a;b):=(-1)^{|a|}E_{1,1}(a;b)$ for each homogeneous element to obtain the desired homotopy. 
The symmetry condition \eqref{Antisym} is satisfied by construction.
The Jacobi relation \eqref{Jacobi}  comes from  \eqref{assoE}.
Indeed, let us first observe that applying \eqref{assoE} we have:
\begin{align*}
E_{1,1}(a;\{b;c\})&=E_{1,1}(a;E_{1,1}(b;c))-(-1)^{(|b|-1)(|c|-1)}E_{1,1}(a;E_{1,1}(c;b))\\
&=E_{1,1}(E_{1,1}(a;b);c)-(-1)^{(|b|-1)(|c|-1)}E_{1,1}(E_{1,1}(a;c);b) +R
\end{align*}
where 
\begin{align*}
 R&=- E_{1,2}(a;b,c)-(-1)^{(|b|-1)(|c|-1)}E_{1,2}(a;c,b)\\
 &-(-1)^{(|b|-1)(|c|-1)}\Bigl( -  E_{1,2}(a;c,b)-(-1)^{(|b|-1)(|c|-1)}E_{1,2}(a;b,c) \Bigr)\\
&= 0.
\end{align*}
From this 
\begin{align*}
\{a;\{b;c\}\}=& E_{1,1}(a;\{b;c\}) -(-1)^{(|a|-1)(|b|+|c|)}E_{1,1}(\{b;c\};a)\\
=&  E_{1,1}(a;\{b;c\})\\ 
&-(-1)^{(|a|-1)(|b|+|c|)}\Bigl(E_{1,1}(E_{1,1}(b;c);a)-(-1)^{(|b|-1)(|c|-1)}E_{1,1}(E_{1,1}(c;b);a) \Bigr)\\
=&E_{1,1}(E_{1,1}(a;b);c)-(-1)^{(|b|-1)(|c|-1)}E_{1,1}(E_{1,1}(a;c);b) \\
&-(-1)^{(|a|-1)(|b|+|c|)}\Bigl(E_{1,1}(b;E_{1,1}(c;a))-(-1)^{(|b|-1)(|c|-1)}E_{1,1}(c;E_{1,1}(b;a)) \Bigr)\\
&+L,
\end{align*}
where
\begin{multline*}
 L=-(-1)^{(|a|-1)(|b|+|c|)}\biggl[ E_{1,2}(b;c,a)+(-1)^{(|c|-1)(|a|-1)}E_{1,2}(b;a,c)\\
 -(-1)^{(|b|-1)(|c|-1)}\Bigl( E_{1,2}(c;b,a)+(-1)^{(|b|-1)(|a|-1)}E_{1,2}(c;a,b) \Bigr)\biggr].
\end{multline*}
By definition
\begin{align*}
- \{\{a;b\};c\}
=&-E_{1,1}(E_{1,1}(a;b);c)+(-1)^{(|a|-1)(|b|-1)}E_{1,1}(E_{1,1}(b;a);c) \\
&+(-1)^{(|c|-1)(|a|+|b|)}\Bigl(E_{1,1}(c;E_{1,1}(a;b))-(-1)^{(|a|-1)(|b|-1)}E_{1,1}(c;E_{1,1}(b;a)) \Bigr).
\end{align*}
Thus we obtain,
\begin{align*}
 \{a;\{b;c\}\} - \{\{a;b\};c\}
=&-(-1)^{(|b|-1)(|c|-1)}E_{1,1}(E_{1,1}(a;c);b) \\
&-(-1)^{(|a|-1)(|b|+|c|)}E_{1,1}(b;E_{1,1}(c;a))\\
&+(-1)^{(|a|-1)(|b|-1)}E_{1,1}(E_{1,1}(b;a);c) \\
&+(-1)^{(|c|-1)(|a|+|b|)}E_{1,1}(c;E_{1,1}(a;b))\\
&+ L.
\end{align*}
Finally, applying once again the equality \eqref{assoE} to both the 3-th and 4-th term, we obtain
\begin{align*}
 \{a;\{b;c\}\} - \{\{a;b\};c\}
=&-(-1)^{(|b|-1)(|c|-1)}E_{1,1}(E_{1,1}(a;c);b) \\
&-(-1)^{(|a|-1)(|b|+|c|)}E_{1,1}(b;E_{1,1}(c;a))\\
&+(-1)^{(|a|-1)(|b|-1)}E_{1,1}(b;E_{1,1}(a;c)) \\
&+(-1)^{(|c|-1)(|a|+|b|)}E_{1,1}(E_{1,1}(c;a);b)\\
&+ L + L' \\
=&-(-1)^{(|b|-1)(|c|-1)}E_{1,1}(\{a;c\};b) 
+(-1)^{(|a|-1)(|b|-1)}E_{1,1}(b;\{a;c\})\\
&+ L + L'\\
=&(-1)^{(|a|-1)(|b|-1)}\{b;\{a;c\}\}\\
&+ L + L',
\end{align*}
where
\begin{multline*}
 L'= (-1)^{(|a|-1)(|b|-1)}\Bigl( E_{1,2}(b;a,c)+(-1)^{(|a|-1)(|c|-1)}E_{1,2}(b;c,a) \Bigr)\\
+(-1)^{(|c|-1)(|a|+|b|)} \Bigl(-E_{1,2}(c;a,b)-(-1)^{(|a|-1)(|b|-1)}E_{1,2}(c;b,a)\Bigr).
\end{multline*}
The equality  $L+L'=0$ is easily verified.\\

The Poisson relation \eqref{Poisson} follows from the equations  \eqref{der11} and \eqref{der21}.
Indeed, take $(-1)^{|b_1|}E_{1,2}(a;b_1,b_2)$ instead of $E_{1,2}(a;b_1,b_2)$ in \eqref{der21}, we obtain:
\begin{align*}
\{a;bc\}=& E_{1,1}(a;bc)-(-1)^{(|a|-1)(|b|+|c|-1)} E_{1,1}(bc;a)\\
\sim & E_{1,1}(a;b)c+ (-1)^{(|a|-1)|b|} bE_{1,1}(a;c)\\
&-(-1)^{(|a|-1)(|b|+|c|-1)}\Bigl(  bE_{1,1}(c;a) + (-1)^{|c|(|a|-1)} E_{1,1}(b;a)c\Bigr)\\
\sim &  \Bigl(E_{1,1}(a;b) -(-1)^{(|a|-1)(|b|+|c|-1)+|c|(|a|-1)} E_{1,1}(b;a)\Bigr)c\\
&+b\Bigl(    (-1)^{(|a|-1)|b|} E_{1,1}(a;c)-(-1)^{(|a|-1)(|b|+|c|-1)}E_{1,1}(c;a) \Bigr)\\
\sim &\{a;b\}c +(-1)^{(|a|-1)|b|}b\{a;c\},
\end{align*}
where $\sim$ is the equivalence relation: $a\sim b$ iff $a$ is homotopic to $b$.

The commutativity between  the bracket and the differential follows from \eqref{diff11}.
Indeed,
\begin{align*}
d_{\La}\{a;b\}=& d_{\La}E_{1,1}(a;b)-(-1)^{(|a|-1)(|b|-1)} d_{\La}E_{1,1}(b;a)\\ 
=&-E_{1,1}(d_{\La}a;b)-(-1)^{|a|}E_{1,1}(a;d_{\La}b)+(-1)^{|a|}\Bigl(a\cdot b -(-1)^{|a||b|}b\cdot a\Bigr)\\
-& (-1)^{(|a|-1)(|b|-1)}\biggl( E_{1,1}(d_{\La}b;a)-(-1)^{|b|}E_{1,1}(b;d_{\La}a)+(-1)^{|b|}\Bigl(b\cdot a -(-1)^{|a||b|}a\cdot b\Bigr)  \biggr)\\
=&-E_{1,1}(d_{\La}a;b)-(-1)^{|a|}E_{1,1}(a;d_{\La}b) \\
&- (-1)^{(|a|-1)(|b|-1)}\biggl( E_{1,1}(d_{\La}b;a)-(-1)^{|b|}E_{1,1}(b;d_{\La}a)\biggr)\\
&+(-1)^{|a|}\Bigl(a\cdot b -(-1)^{|a||b|}b\cdot a\Bigr) - (-1)^{(|a|-1)(|b|-1)+|b|}\Bigl(b\cdot a -(-1)^{|a||b|}a\cdot b\Bigr)\\
=&-E_{1,1}(d_{\La}a;b)-(-1)^{|a|}E_{1,1}(a;d_{\La}b)\\
& - (-1)^{(|a|-1)(|b|-1)}\biggl( E_{1,1}(d_{\La}b;a)-(-1)^{|b|}E_{1,1}(b;d_{\La}a)\biggr)\\
=&- \{d_{\La}a;b\} - (-1)^{|a|}\{a;d_{\La}b\}.
\end{align*}
\end{proof}
\begin{prop}\label{prop: morphismes G-alg Gerst}
 {An $\infty$-morphism of homotopy G-algebras induces in homology a morphism of Gerstenhaber algebras.}
\end{prop}
\begin{proof}
Let us take the following convention $f_i:=\widetilde{f}_i(\s^{\ot i})$. Then the equation \eqref{eqmorph hGa1} reads for $k=l=1$:
\begin{equation}\label{eq: morphisme Galg homot f1 et f2}
E_{1,1}(f_1(a);f_1(b))-f_1E_{1,1}(a;b)= (-1)^{|a|-1}f_2(a;b) +(-1)^{|a|(|b|-1)}f_2(b;a).
\end{equation}
We obtain
\begin{align*}
 \{f_1(a);f_1(b)\}=&~E_{1,1}(f_1(a);f_1(b))-(-1)^{|a|-1)(|b|-1)}E_{1,1}(f_1(b);f_1(a))\\
=&~ f_1E_{1,1}(a;b) +(-1)^{|a|-1}f_2(a;b) -(-1)^{|a|(|b|-1)}f_2(b;a)\\
& -(-1)^{|a|-1)(|b|-1)}\Bigl(f_1E_{1,1}(b;a) +(-1)^{|b|-1}f_2(b;a) -(-1)^{|b|(|a|-1)}f_2(a;b)\Bigr)\\
=&~  f_1(\{a;b\}).
\end{align*}
\end{proof}

\paragraph{Acknowledgement}

I would like to thank Jean-Claude Thomas for his very helpful remarks, in particular for suggesting the study of suspensions.
I am grateful to Muriel Livernet for the useful conversation at IHP and her many comments about the content of this paper.

\bibliography{bibliothese4b2-UTF8}
\bibliographystyle{plain}
\end{document}